\documentclass[12pt,a4paper,fleqn]{article}
\usepackage{exscale,german,amsthm}
\usepackage{dsfont}
\usepackage[intlimits]{amsmath}
\usepackage[cp850]{inputenc}
\usepackage{epsfig}
\usepackage{stmaryrd}
\usepackage{color}
\usepackage{enumerate}
\usepackage{amssymb}
\usepackage{textcomp}
\usepackage{natbib}
\usepackage{pdflscape}
\newcommand{\bol}[1]{\mbox{\boldmath$#1$}}

\newcommand{\bSigma}{\bol{\Sigma}}

\newcommand{\bmu}{\bol{\mu}}
\newcommand{\bzeta}{\bol{\zeta}}

\newcommand{\bnu}{\bol{\nu}}

\newcommand{\bOmega}{\bol{\Omega}}
\newcommand{\bUpsilon}{\bol{\Upsilon}}
\newcommand{\bXi}{\bol{\Xi}}

\newcommand{\bm}{\mathbf{m}}
\newcommand{\bv}{\mathbf{v}}

\newcommand{\bb}{\mathbf{b}}

\newcommand{\bX}{\mathbf{X}}
\newcommand{\bNull}{\mathbf{0}}
\newcommand{\bOne}{\mathbf{1}}
\newcommand{\bx}{\mathbf{x}}

\newcommand{\by}{\mathbf{y}}

\newcommand{\bM}{\mathbf{M}}

\newcommand{\bI}{\mathbf{I}}
\newcommand{\bS}{\mathbf{S}}
\newcommand{\bL}{\mathbf{L}}
\newcommand{\bZ}{\mathbf{Z}}
\newcommand{\bl}{\mathbf{l}}

\newcommand{\bz}{\mathbf{z}}
\newcommand{\bU}{\mathbf{u}}
\newcommand{\bV}{\mathbf{V}}

\newcommand{\bw}{\mathbf{w}}

\newcommand{\bi}{\mathbf{1}}

\newtheorem{theorem}{Theorem}

\newtheorem{proposition}{Proposition}
\newtheorem{lemma}{Lemma}

\newfont{\tabfont}{cmr7 at 7pt}
\begin{document}

\begin{center}
\noindent\textbf{\large Bayesian Inference of the Multi-Period Optimal Portfolio for an Exponential Utility}\vspace*{1cm}

\textsc{David Bauder$^a$, Taras Bodnar$^{b,1}$\footnote{$^1$Corresponding Author: Taras Bodnar. E-Mail: taras.bodnar@math.su.se. Tel: +46 8 164562. Fax: +46 8 612 6717. This research was partly supported by the German Science Foundation (DFG) via the projects BO 3521/3-1 and SCHM 859/13-1 ''Bayesian Estimation of the Multi-Period Optimal Portfolio Weights and Risk Measures''.}, Nestor Parolya$^c$, Wolfgang Schmid$^d$} \vspace*{0.2cm}
{\it \footnotesize \\
$^a$ Department of Mathematics, Humboldt-University of Berlin, D-10099 Berlin, Germany\\
$^b$ Department of Mathematics, Stockholm University, SE-10691 Stockholm, Sweden\\
$^c$ Institute of Statistics, Leibniz University Hannover, D-30167 Hannover, Germany\\
$^d$ Department of Statistics, European University Viadrina, PO Box 1786,  15207 Frankfurt (Oder), Germany
}
\end{center}
\vspace{0.2cm}
 \begin{abstract}
 We consider the estimation of the multi-period optimal portfolio obtained by maximizing an exponential utility. Employing Jeffreys' non-informative prior and the conjugate informative prior, we derive stochastic representations for the optimal portfolio weights at each time point of portfolio reallocation. This provides a direct access not only to the posterior distribution of the portfolio weights but also to their point estimates together with uncertainties and their asymptotic distributions. Furthermore, we present the posterior predictive distribution for the investor's wealth at each time point of the investment period in terms of a stochastic representation for the future wealth realization. This in turn makes it possible to use quantile-based risk measures or to calculate the probability of default. We apply the suggested Bayesian approach to assess the uncertainty in the multi-period optimal portfolio by considering assets from the FTSE 100 in the weeks after the British referendum to leave the European Union. The behaviour of the novel portfolio estimation method in a precarious market situation is illustrated by calculating the predictive wealth, the risk associated with the holding portfolio, and the default probability in each period.
 \end{abstract}

\vspace{0.5cm}


%
\newpage

\section{Introduction}

In portfolio theory, the mean-variance paradigm introduced by \cite{Markowitz1952} is still a popular reference for understanding the relationship between systematic risk, return and investment behaviour. A portfolio is determined here by using the asset expected returns and their covariances. As a starting point, \cite{Markowitz1952} was vastly extended in the following 70 years. While \cite{Markowitz1952} focused only on a single investment period, the multi-period solution was introduced in \cite{Markowitz1959}. Merton (1969) showed that the mean-variance multi-period setting in the continuous time case is equivalent to expected utility maximization for an exponential utility function. The multi-period optimal portfolio choice problems for different utility functions were considered by \cite{Mossin1968}, \cite{Samuelson1969}, \cite{EltonGruber1974}, \cite{BrandtSanta-Clara2006}, \cite{Basak2010}.

While these studies focus on the continuous time case, \cite{li2000optimal}, \cite{ccanakouglu2009portfolio}, \cite{bodnar2015closed,bodnar2015exact} presented the results in the discrete time case for the quadratic utility function and the exponential utility function. In particular, \cite{bodnar2015exact} derived an analytical expression for the multi-period optimal portfolio weights  under the assumption of non-tradable predictable variables and a VAR(1)-structure which are described as linear combinations of the precision matrix (inverse covariance matrix) and the expected return vector. While this setting allows for flexibility in building trading strategies under quite unrestrictive assumptions, there are still shortcomings: (i) since the parameters of the asset return distribution, namely the mean vector and the covariance matrix, are unknown quantities, the optimal portfolio weights cannot be constructed in practice and they are obtained by replacing the unknown parameter of the asset return distribution by the corresponding estimates; (ii) although the distributional properties of the estimated optimal portfolio weights and corresponding inference procedures were derived in a number of literature studies for the single-period investment strategies (see, e.g., \cite{gibbons1989test}, \cite{shanken1992estimation}, \cite{shanken2007estimating}, \cite{OkhrinSchmid2006}, \cite{BodnarSchmid2008, BodnarSchmid2011}, \cite{bodnar2009econometrical}), the problem with the overlapping estimation windows appears to be very crucial under the multi-period setting; (iii) due to the multivariate structure, the determination of the joint distribution of the estimated multi-period optimal portfolio weights is a challenging task.

To tackle all these three challenges, we opt for a Bayesian approach. The Bayesian approach is a well established method for building trading strategies in a single-period optimal portfolio choice problem, starting with \cite{Winkler1973} and \cite{WinklerBarry1975} and continued  until this day. For an overview, see, e.g., \cite{Brandt2010} where also Bayesian portfolio methods are discussed, or \cite{AvramovZhou2010}. As \cite{AvramovZhou2010} pointed out, the Bayesian setting is a realistic description of human decision making processes and information utilization. Both past events and experiences influence the beliefs of market participants at least up to a certain degree how an investment will develop. The investor beliefs are modeled via a prior distributions which represents the relevant information regarding the behaviour of the asset returns. While there is a plenty of possibilities to specify the prior, we focus on the non-informative diffuse prior and the informative conjugate prior (see, e.g., \cite{Zellner1971}, and \cite{gelman2014bayesian}) not only for computational reasons but mainly because of their popularity in the financial literature (c.f., \cite{Barry1974}, \cite{Brown1976}, \cite{KleinBawa1976}, \cite{FrostSavarino1986}, \cite{aguilar2000bayesian}, \cite{Rachev2008}, \cite{AvramovZhou2010}, \cite{Sekerke2015}, \cite{BodnarMazurOkhrin2016}). Furthermore, their application allows to derive the corresponding posterior distributions in the closed-form what enables us to access important risk measures and to construct credible sets.

The obtained posterior distributions of the optimal portfolio weights under both employed priors are presented in terms of their stochastic representations. A stochastic representation is a well established tool in computational statistics (c.f., \cite{givens2012computational}) and in the theory of elliptically contoured distributions (see, e.g. \cite{Gupta2013}) which was already used in Bayesian statistics by \cite{BodnarMazurOkhrin2016}. It turns out that the derived stochastic representations are very powerful, allowing us to access not only the posterior distribution of the multi-period optimal portfolio weights, but also to determine the predictive distribution for the wealth at each point of the holding period. Therefore, we are able to access the quantiles for the posterior predictive wealth distribution and can calculate the risk associated with the portfolio at every point over the lifetime of a portfolio, besides analytical Bayesian estimates for the weights together with their uncertainties. Besides these pleasing properties, the developed stochastic representations are highly efficient from a computational point of view since Markov-Chain Monte-Carlo methods are not longer needed. In addition to the derivation of these results, we illustrate this method and its properties on real data. We test the model in an exhaustive study using data from the FTSE 100, where the portfolios cover the time of Great Britains referendum to leave the European Union on 23.6.2016, more commonly regarded as ``Brexit'', where a slim majority of British voters decided to leave the European Union. Although this result was regarded as the less likely option in advance, it was regarded as the option with the least favourable effects on the British economy and should therefore have an effect on a portfolio covering this period.

The remaining paper is structured in the following way. In Section 2, we briefly review the solution of the multi-period optial portfolio choice problem with exponential utility derived in \cite{bodnar2015exact}. The stochastic representations for the optimal portfolio weights under both priors are presented in Theorems \ref{stochrep} and \ref{stochrep2} (Section 2.2), which are use to derive the corresponding Bayes estimates for the weights (Theorem \ref{th3}) together with their covariance matrix (Theorem \ref{th4}) as well as to prove the posterior asymptotic normality (Theorem \ref{th5}). In Section 2.3, we obtain the posterior predictive distribution for the wealth during the holding period which is provided in terms of stochastic representation in Theorem \ref{th6} under both employed priors. In Section 3, the suggested Bayesian approach is applied to the Brexit-data by calculating the asymptotic distributions for the optimal portfolio weights, determining the credible sets for the portfolio wealth and specifying the default probabilities at each time point. Section 4 summarizes the main results of the paper, while all technical proofs are moved to the appendix (Section 5).

\section{Bayesian analysis of multi-period optimal portfolios}
\subsection{Analytical solution of the multi-period optimization problem}
Let $\bX_t = (X_{t,1},X_{t,2},...,X_{t,k})^\top$ be a random vector of returns on $k$ assets taken at time point $t$. Throughout the paper we assume that the asset returns $\bX_1, \bX_2, ...$ are infinitely exchangeable and multivariate centered spherically symmetric. This assumption, in particular, implies (see, e.g., \citet[Proposition 4.6]{BernardoSmith2000}) that the asset returns are independently and identically distributed given mean vector $\bmu$ and covariance matrix $\bSigma$ with the conditional distribution given by $\bX_t|\bmu,\bSigma \sim \mathcal{N}_k (\bmu,\bSigma)$ ($k$-dimensional normal distribution with mean vector $\bmu$ and covariance matrix $\bSigma$). It is noted that the imposed assumption imply that neither the unconditional distribution of the asset returns is normal nor that they are independently distributed. Moroever, the unconditional distribution of the asset returns appears to be heavy-tailed which is usually observed for financial data.

The quantities $\bmu$ and $\bSigma$ denote the parameters of the asset returns distribution where $\bSigma$ is assumed to be a $k\times k$ dimensional positive definite matrix. We consider a multi-period portfolio choice problem with the allocation of initial wealth at time point $t=0$ and with the subsequent update of the portfolio structure at time points $t \in \{1,2,...,T\}$. Let $\bv_t = (v_{t,1}, ...,v_{t,k})^\top$ stand for the vector of portfolio weights determined at time $t$ and let $r_{f,t}$ be the return on the risk-free asset in period $t$. We assume that short-selling is allowed, i.e. the weights could also be negative. The vector $\bv_t$ specifies the structure of the portfolio related to the risky assets, whereas the part of the wealth equal to $1-\bi^\top \bv_t$ is invested into the risk-free asset where $\bi$ denotes the $k$-dimensional vector of ones. Then the investor's wealth in period $t$ is expressed as
\begin{eqnarray*}
	W_t &=& W_{t-1}(1+(1-\bOne^\top \bv_{t-1})r_{f,t} + \bv_{t-1} ^\top \bX_t)= W_{t-1}(1+r_{f,t} + \bv_{t-1} ^\top (\bX_t-r_{f,t}\bOne)).
\end{eqnarray*}

An investor seeks to maximize the utility of the final wealth, i.e. $U(W_T)$, where $U(x)=-\exp(-\gamma x)$ is the exponential utility function and the coefficient of absolute risk aversion, $\gamma>0$, determines the investor's attitude towards risk. The optimization problem is given by
\begin{equation} \label{opt_prob}
V(0,W_0) =	\max_{\{\bv_s\}_{s=0} ^{T-1}}\mathbb{E}_0[U(W_T)]
\end{equation}
where the maximum is taken with respect to all weights $\bv_0$,..., $\bv_{T-1}$ which specify the portfolio structure during the initial period of investment as well as during all consequent reallocations. The solution of \eqref{opt_prob} is derived in the recursive way starting from the last period by applying Bellman equations at $0$, $1$, ... $T-1$. The optimization problem at time point $T-t$ is then given by
\begin{eqnarray*}
		V(T-t,W_{T-t}) &=& \max_{\{\bv_s\}_{s=T-t}^{T-1}}\mathbb{E}_{T-t}\left[	\max_{\{\bv_s\}_{s=T-t+1} ^{T-1}}\mathbb{E}_{T-t+1}[U(W_T)]\right] \ \\
		&=& \max_{\bv_{T-t}} \mathbb{E}_{T-t} \left[V(T-t+1,W_{T-t} \left(r_{f,T-t}+\bw_{T-t+1}^{\top} (\bX_{T-t+1}-r_{f,T-t+1}\bOne) \right)) \right]
\end{eqnarray*}
subject to the terminal condition $U(W_T) = -\exp(-\gamma W_T)$ with $\bw_{T-t+1}$ as the optimal portfolio weights in period $T-t+1$. For details on this method, see e.g. \cite{pennacchi2008theory}, while \cite{bodnar2015exact} determine an analytical solution of \eqref{opt_prob} under the exponential utility. The latter results are summarized in Proposition \ref{portfolio-solution}.

\begin{proposition} \label{portfolio-solution}
Let $\bX_t$, $t=0,...,T$ be a sequence of conditionally independently and identically distributed vectors of $k$ risky assets with $\bX_t|\bmu,\bSigma \sim \mathcal{N}_k(\bmu,\bSigma)$. Let $\bSigma$ be positive definite. Then the optimal multi-period portfolio weights are given by
	\begin{eqnarray}\label{MPP_sol}
	\bw_{t}  &=& C_{t}\bSigma^{-1}(\bmu - r_{f,t+1}\bOne),\quad \text{with}\quad C_{t}=(\gamma W_{t} \prod_{i=t+2} ^T R_{f,i})^{-1}
	\end{eqnarray}
for $t=0,...,T-1$ where $R_{f,i}=1+r_{f,i}$ and  $\prod_{i=T+1} ^T R_{f,i} \equiv 1$.
\end{proposition}

Although Proposition \ref{portfolio-solution} provides a simple solution of the multi-period portfolio choice problem, the formula \eqref{MPP_sol} cannot directly be applied in practice since $\bmu$ and $\bSigma$ are unknown parameters of the asset return distribution. As a result, these two quantities have to be estimated before the portfolio \eqref{MPP_sol} is constructed. However, the usage the estimated mean vector and the estimated covariance matrix instead of the population ones does not ensure that the estimated portfolio weights coincide with true ones. Then two main questions raise: (i) how strongly deviates the estimated portfolio from the population one? and (ii) is it reasonable to invest into the estimated portfolio? Both questions have to be treated by using statistical methods and are very closely connected to the distributional properties of the estimates constructed for $\bmu$ and $\bSigma$.

The traditional approach of estimating the portfolio weights relies on the methods from the conventional statistics where the sample mean vector and the sample covariance matrix are used. Let $\bx_{t-n+1}, ..., \bx_{t}$ be the observation vectors of asset returns which are considered as realizations of the corresponding random vectors $\bX_i$, $i=t-n+1,...,t$. Then the mean vector and the covariance matrix at time point $t$ are estimated by
\begin{equation}\label{sample_mean_cov}
\overline{\bx}_{t}=\frac{1}{n} \sum_{i=t-n+1}^{t} \bx_i \quad \text{and} \quad \bS_t=\frac{1}{n-1} \sum_{i=t-n+1}^{t}  (\bx_i-\overline{\bx}_{t})(\bx_i-\overline{\bx}_{t})^\top\,.
\end{equation}
The sample estimate of the multi-period optimal portfolio is obtained by replacing $\bmu$ and $\bSigma$ in \eqref{MPP_sol} by the corresponding estimates from \eqref{sample_mean_cov}. This leads to
\begin{equation}\label{sample_port}
\hat{\bw}_t=C_{t}\bS_t^{-1}(\overline{\bx}_{t} - r_{f,t+1}\bOne)\quad \text{with}\quad C_{t}=(\gamma W_{t} \prod_{i=t+2} ^T R_{f,i})^{-1} ~~ \text{for} ~~
t=0,...,T-1.
\end{equation}

Using the findings in \cite{BodnarOkhrin2011}, we obtain the density function, the moments and the stochastic representation of the sample multi-period optimal portfolio weights from the viewpoint of frequentist statistics. These results provide answers on the above two questions and allow us to characterize the distributional properties of each vector of weights $\hat{\bw}_t$ separately. On the other hand, they do not take into account the multi-period nature of the considered investment procedure. More precisely, it is not possible to provide the characterization of the whole multi-period optimal portfolio, since the overlapping samples are used and the dependence structure between the estimated portfolio weights becomes severe.

For that reason, we deal with the problem of estimating the multi-period optimal portfolio from the viewpoint of Bayesian statistics and consider the portfolio constructed by using \eqref{sample_port} as a benchmark portfolio without investigating its distributional properties in detail. In contrast to the methods of the frequentist statistics, the application of the Bayesian approach allows the sequential update of the available information which is a very important property needed for estimating the multi-period portfolio weights.

\subsection{Bayesian estimation of portfolio weights}

\noindent Let $\bx_{t,n}=(\bx_{t-n+1},...,\bx_t)$ denote the observation matrix at time point $t$ which consists of $n$ asset return vectors from $t-n+1$ to $t$. According to Bayes theorem, the beliefs regarding $\bmu$ and $\bSigma$ are updated in the presence of occurring data, yielding the posterior distribution $\pi(\bmu,\bSigma|\bx_{t,n})$ to be proportional to the product of the likelihood function $L(\bx_{t,n}|\bmu,\bSigma)$ and the prior distribution $\pi(\bmu,\bSigma)$. The posterior is, then, used to derive  Bayesian estimates for the multi-period optimal portfolio weights as well as their characteristics, like the covariance matrix and a credible region which is an analogue to a confidence region in the conventional statistics. The Bayes theorem states that
\begin{eqnarray*}
	\pi(\bmu,\bSigma|\bx_{t,n}) &\propto& L(\bx_{t,n}|\bmu,\bSigma) \pi(\bmu,\bSigma).
\end{eqnarray*}

The choice of the prior $\pi(\bmu,\bSigma)$ is an important step in the Bayesian decision process. Although the prior should reflect the investor's belief regarding the parameters of the asset return distribution, it also strongly affects the model's computational properties since it influences the accessibility of the posterior distribution. Several priors for the mean vector and covariance matrix of the asset returns have been suggested in literature (see, e.g., \cite{Barry1974}, \cite{Brown1976}, \cite{KleinBawa1976}, \cite{FrostSavarino1986}, \cite{Rachev2008}, \cite{AvramovZhou2010}, \cite{Sekerke2015}) with the recent paper of \cite{BodnarMazurOkhrin2016} summarizing these results. In the following, we choose Jeffreys' non-informative prior and a conjugate informative prior for both $\bmu$ and $\bSigma$. These two priors are widely used in the context of Bayesian inference of optimal portfolios.

The Jeffreys non-informative prior, also known as the diffuse prior, is given by
\begin{eqnarray} \label{noninformativecor1}
	\pi(\bmu,\bSigma) &\propto& |\bSigma| ^{-(k+1)/2}
\end{eqnarray}
while the cojugate prior is expressed as
\begin{eqnarray} \label{informativenucor1}
\bmu|\bSigma &\sim& \mathcal{N}_k\left(\bm_0,\frac{1}{r_0}\bSigma\right)\text{,}\\
 \bSigma &\sim& \mathcal{IW}_k(d_0,\bS_0),\label{informativesigmacor1}
\end{eqnarray}
where $\bm_0$, $r_0$, $d_0$, $\bS_0$ are additional model parameters known as hyperparameters. The symbol $\mathcal{IW}_k(d_0,\bS_0)$ denotes the inverse Wishart distribution with $d_0$ degrees of freedom and parameter matrix $\bS_0$. The prior mean $\bmu_0$ reflects our prior expectations about the expected asset returns, while $\bS_0$ presents in the model the prior beliefs about the covariance matrix. The other two hyperparameters $r_0$ and $d_0$ are known as precision parameters for $\bmu_0$ and $\bS_0$, respectively.
Note that the prior \eqref{informativenucor1}-\eqref{informativesigmacor1} corresponds to the well-known conjugate normal-inverse-Wishart model as discussed by, e.g., \cite{gelman2014bayesian}. In this case the posterior is accessible in an analytical form and moreover, has the same distribution as the prior with updated hyperparameters.

In Proposition \ref{prop2}, we present the marginal posterior of $\bmu$ as well as the conditional posterior of $\bSigma$ given $\bmu$. These results will be later used in the derivation of Bayesian estimates for the optimal portfolio weights. In the following the symbol $t_k(d,\mathbf{a},\mathbf{A})$ stands for the multivariate $k$-dimensional $t$-distribution with $d$ degrees of freedom, location vector $\mathbf{a}$ and dispersion matrix $\mathbf{A}$. In the case of $k=1$, $\mathbf{a}=0$, and $\mathbf{A}=1$, we use the notation $t_d$ to denote the standard univariate $t$-distribution with $d$ degrees of freedom.

\begin{proposition}\label{prop2}
Let $\bX_{t-n+1},...,\bX_t$ be conditionally independently distributed with $\bX_i|\bmu,\bSigma \sim \mathcal{N}_k(\bmu,\bSigma)$ for $i =t-n+1,...,t$ with $n>k$. Then:
\begin{enumerate}[(a)]
\item Under the diffuse prior (\ref{noninformativecor1}), the marginal posterior distribution of $\bmu$ is given by
	\begin{eqnarray*}
	\bmu|\bx_{t,n} &\sim& t_k\left(n-k, \overline{\bx}_{t,d}, \frac{1}{n(n-k)}\bS_{t,d} \right) ~~ \text{with} ~~\overline{\bx}_{t,d}=\overline{\bx}_{t}
~~ \text{and} ~~\bS_{t,d}=(n-1)\bS_t.
	\end{eqnarray*}
		
The conditional posterior distribution of $\bSigma$ given $\bmu$ is expressed as
	\begin{eqnarray*}
	\bSigma|\bmu,\bx_{t,n} &\sim& \mathcal{IW}_k(n+k+1,\bS_{t,d}^*(\bmu)) ~~ \text{with} ~~\bS_{t,d}^*(\bmu) = \bS_{t,d}+n(\bmu-\overline{\bx}_{t,d})(\bmu-\overline{\bx}_{t,d})^\top.
\end{eqnarray*}
\item Under the conjugate prior $(\ref{informativenucor1})$ and $(\ref{informativesigmacor1})$, the marginal posterior distribution of $\bmu$ is given by
	\begin{eqnarray*}
		\bmu|\bx_{t,n} &\sim& t_k\left(n+d_0-2k, \overline{\bx}_{t,c}, \frac{1}{(n+r_0)(n+d_0-2k)} \bS_{t,c}  \right) ~~\text{with}
	\end{eqnarray*}
	\begin{eqnarray*}
	\overline{\bx}_{t,c} = \frac{n\overline{\bx}_{t}+r_0\bm_0}{n+r_0}~~\text{and}~~ \bS_{t,c}=\bS_{t,d}+\bS_0+nr_0\frac{(\bm_0-\overline{\bx}_{t,c})(\bm_0-\overline{\bx}_{t,c})^\top}{n+r_0} .
	\end{eqnarray*}
The conditional posterior distribution of $\bSigma$ given $\bmu$ is expressed as
	\begin{eqnarray*}
		\bSigma|\bmu,\bx_{t,n} &\sim& IW_k(n+d_0+1,\bS^*_{t,c}(\bmu)) \quad \text{with}\\
	\bS^*_{t,c}(\bmu) &=& \bS_{t,c} + (n+r_0)(\bmu-\overline{\bx}_{t,c})(\bmu-\overline{\bx}_{t,c})^\top.
\end{eqnarray*}
\end{enumerate}
\end{proposition}

The proof of Proposition \ref{prop2} follows from chapter 3 in \cite{gelman2014bayesian} who presented the expressions of the marginal posterior distributions of $\bmu$ under both the diffuse and the conjugate priors. Then, the results for the conditional posteriors of $\bSigma$ are obtained from the joint posterior distributions using the formulae for the marginal posteriors for $\bmu$. It is remarkable that although the results for the marginal posteriors for both $\bmu$ and $\bSigma$ are widely used in Bayesian inferences and the conditional posteriors for $\bmu$ given $\bSigma$ have been considered previously in literature (see, e.g., \cite{SunBerger2007}), the results for the conditional posteriors of $\bSigma$ given $\bmu$ have not been discussed nor used. Next, we show that the last finding allows to derive posterior distributions for functions which includes both $\bmu$ and $\bSigma$.

In order to assess the risk associated with estimating the optimal portfolio weights, we need to derive results about the posterior distribution of the weights presented in Proposition \ref{portfolio-solution} which are given as a product of the inverse covariance matrix and the mean vector. Next, we establish very useful stochastic representations for these weights, endowing the parameters with their diffuse and conjugate priors. The results are summarized in Theorem \ref{stochrep}, where the stochastic representations are derived for an arbitrary linear combination of optimal portfolio weights. These findings are later used for calculating the Bayesian estimates of the portfolio weights (Theorem \ref{th3}) and their covariance matrix (Theorem \ref{th4}). It is noted that the application of the stochastic representation to describe the distribution of random quantities has been used both in the conventional statistics (see, e.g., \cite{givens2012computational}, \cite{Gupta2013}) and the Bayesian statistics (c.f., \cite{BodnarMazurOkhrin2016}). Later on, the symbol ''$\stackrel{d}{=}$'' denotes the equality in distribution. The proof of Theorem \ref{stochrep} is presented in the appendix (Section 5).

\begin{theorem}\label{stochrep}
Let $\bL$ be a $p\times k$-dimensional matrix of constants. Then under the assumption of Proposition \ref{prop2} we get:
\begin{enumerate}[(a)]
\item Under the diffuse prior (\ref{noninformativecor1}), the stochastic representation of $\bL \bw_t$ is given by
\begin{eqnarray*}
\bL \bw_t &\stackrel{d}{=}& C_t \eta \bL \bS_{t,d}^*(\bmu)^{-1} (\bmu-r_{f,t+1})
+ C_t \sqrt{ \eta } \Big((\bmu-r_{f,t+1})^\top \bS_{t,d}^*(\bmu)^{-1} (\bmu-r_{f,t+1}) \cdot \bL \bS_{t,d}^*(\bmu)^{-1} \bL^\top \\
&-& \bL \bS_{t,d}^*(\bmu)^{-1} (\bmu-r_{f,t+1})(\bmu-r_{f,t+1})^\top \bS_{t,d}^*(\bmu)^{-1} \bL^\top \Big)^{1/2} \bz_0,
\end{eqnarray*}
where $\eta  \sim \chi^2_{n}$, $\bz_0 \sim \mathcal N _p (\mathbf 0 , \mathbf I_p)$, and $\bmu | \bx \sim t_k\left(n-k,\overline{\bx}_{t,d},\bS_{t,d}/(n(n-k))\right)$; moreover, $\eta, \mathbf z_0$ and $\bmu$ are mutually independent.

\item Under the conjugate prior $(\ref{informativenucor1})$ and $(\ref{informativesigmacor1})$, the stochastic representation of $\bL \bw_t$ is given by
\begin{eqnarray*}
\bL \bw_t &\stackrel{d}{=}& C_t \eta \bL \bS_{t,c}^*(\bmu)^{-1} (\bmu-r_{f,t+1})
+ C_t \sqrt{ \eta } \Big((\bmu-r_{f,t+1})^\top \bS_{t,c}^*(\bmu)^{-1} (\bmu-r_{f,t+1}) \cdot \bL \bS_{t,c}^*(\bmu)^{-1} \bL^\top \\
&-& \bL \bS_{t,c}^*(\bmu)^{-1} (\bmu-r_{f,t+1})(\bmu-r_{f,t+1})^\top \bS_{t,c}^*(\bmu)^{-1} \bL^\top \Big)^{1/2} \bz_0,
\end{eqnarray*}
where $\eta  \sim \chi^2_{n+d_0-k}$, $\bz_0 \sim \mathcal N _p (\mathbf 0 , \mathbf I_p)$, and $\bmu | \bx \sim t_k\left(n+d_0-2k,\overline{\bx}_{t,c},\bS_{t,c}/((n+r_0)(n+d_0-2k))\right)$; moreover, $\eta, \mathbf z_0$ and $\bmu$ are mutually independent.
\end{enumerate}
\end{theorem}

The results of Theorem \ref{stochrep} show that in both cases, i.e., when the mean vector and the covariance matrix are endowed by the diffuse prior and the conjugate prior, the obtained stochastic representations are very similar and the posterior distributions of the multi-period optimal portfolio weights from Proposition \ref{portfolio-solution} can be described by three random variables which have standard univariate/multivariate distributions.

Another important application of Theorem \ref{stochrep} is that the results of this theorem also provide a hint how these distributions can be accessed in practice via simulations, namely by simulating samples from the $\chi^2$-distribution, the normal distribution, and the $t$-distribution. Although the derived stochastic representations have some nice computational properties in terms of speed, they are not computationally efficient. In the following theorem we derive further stochastic representations under both priors by applying the Sherman-Morrison-Woodbury formula on the inverse of the posterior scale matrices $\bS_{t,d}^*(\bmu)$ and  $\bS^*_{t,c}(\bmu)$. The proof of the theorem is provided in the appendix. Let $\mathcal{F}(d_1,d_2)$ denote the $F$-distribution with $d_1$ and $d_2$ degrees of freedom.
	
\begin{theorem}\label{stochrep2}
Under the assumption of Theorem \ref{stochrep} we get:
\begin{enumerate}[(a)]
\item Under the diffuse prior (\ref{noninformativecor1}), the stochastic representation of $\bL \bw_t$ is given by
\begin{eqnarray}\label{th2_eq1}
\bL \bw_t &\stackrel{d}{=}& C_t \eta \bL \bzeta_d +
C_t\sqrt{\eta}\left(\epsilon_d \bL \bUpsilon_d \bL^\top -\bL \bzeta_d \bzeta_d^\top \bL^\top\right)^{1/2} \bz_0,
 \end{eqnarray}
with
\begin{eqnarray*}
\epsilon_d &=& \epsilon_d (Q, \bU) = (\overline{\bx}_{t,d}-r_{f,t+1}\bOne)^\top \bS_{t,d}^{-1} (\overline{\bx}_{t,d}-r_{f,t+1}\bOne)\\
&+& \frac{2}{\sqrt{n}}\frac{\sqrt{k Q/(n-k)}}{1+k Q/(n-k)}(\overline{\bx}_{t,d}-r_{f,t+1}\bOne)^\top \bS_{t,d}^{-1/2} \bU\\
& +& \frac{1}{n}\frac{k Q/(n-k)}{1+ k Q/(n-k)}
-\frac{k Q/(n-k)}{1+k Q/(n-k)}\left( (\overline{\bx}_{t,d}-r_{f,t+1}\bOne)^\top \bS_{t,d}^{-1/2} \bU\right)^2,
\\
\bzeta_d&=&\bzeta_d (Q, \bU)= \bS_{t,d}^{-1} (\overline{\bx}_{t,d}-r_{f,t+1}\bOne)+ \frac{1}{\sqrt{n}}\frac{\sqrt{k  Q/(n-k)} }  { 1+ k Q/(n-k)}\bS_{t,d}^{-1/2} \bU\\
&-&\frac{k Q/(n-k) }  {   1+ k Q/(n-k)} \bS_{t,d}^{-1/2} \bU\bU^\top \bS_{t,d}^{-1/2} (\overline{\bx}_{t,d}-r_{f,t+1}\bOne),
\\
\bUpsilon_d &=&\bUpsilon_d (Q, \bU)=\bS_{t,d}^{-1}-
\frac{k Q/(n-k) }  {   1+ k Q/(n-k)} \bS_{t,d}^{-1/2} \bU\bU^\top \bS_{t,d}^{-1/2},
\end{eqnarray*}
where $\eta  \sim \chi^2_{n}$, $\bz_0 \sim \mathcal{N}_p (\mathbf{0} , \bI_p)$, $Q \sim \mathcal{F}(k,n-k)$, and $\bU$ uniformly distributed on the unit sphere in $\mathds{R}^k$; moreover, $\eta$, $\bz_0$, $Q$, and $\bU$ are mutually independent.

\item Under the conjugate prior $(\ref{informativenucor1})$ and $(\ref{informativesigmacor1})$, the stochastic representation of $\bL \bw_t$ is given by
\begin{eqnarray}\label{th2_eq2}
\bL \bw_t &\stackrel{d}{=}& C_t \eta \bL \bzeta_c +
 C_t\sqrt{\eta}\left(\epsilon_c \bL \bUpsilon_c \bL^\top -\bL \bzeta_c \bzeta_c^\top \bL^\top\right)^{1/2} \bz_0,
 \end{eqnarray}
with
\begin{eqnarray*}
\epsilon_c &=& \epsilon_d (Q, \bU) = (\overline{\bx}_{t,c}-r_{f,t+1}\bOne)^\top \bS_{t,d}^{-1} (\overline{\bx}_{t,c}-r_{f,t+1}\bOne)\\
&+& \frac{2}{\sqrt{n+r_0}}\frac{\sqrt{k Q/(n+d_0-2k)}}{1+k Q/(n+d_0-2k)}(\overline{\bx}_{t,c}-r_{f,t+1}\bOne)^\top \bS_{t,d}^{-1/2} \bU\\
 &+& \frac{1}{n+r_0}\frac{k Q/(n+d_0-2k)}{1+ k Q/(n+d_0-2k)}
-\frac{k Q/(n+d_0-2k)}{1+k Q/(n+d_0-2k)}\left( (\overline{\bx}_{t,c}-r_{f,t+1}\bOne)^\top \bS_{t,d}^{-1/2} \bU\right)^2,
\\
\bzeta_c&=&\bzeta_d (Q, \bU)= \bS_{t,c}^{-1} (\overline{\bx}_{t,c}-r_{f,t+1}\bOne)+ \frac{1}{\sqrt{n+r_0}}\frac{\sqrt{k  Q/(n+d_0-2k)} }  { 1+ k Q/(n+d_0-2k)}\bS_{t,c}^{-1/2} \bU\\
&-&\frac{k Q/(n+d_0-2k) }  {   1+ k Q/(n+d_0-2k)} \bS_{t,c}^{-1/2} \bU\bU^\top \bS_{t,c}^{-1/2} (\overline{\bx}_{t,c}-r_{f,t+1}\bOne),
\\
\bUpsilon_c &=&\bUpsilon_d (Q, \bU)=\bS_{t,c}^{-1}-
\frac{k Q/(n+d_0-2k) }  {   1+ k Q/(n+d_0-2k)} \bS_{t,c}^{-1/2} \bU\bU^\top \bS_{t,c}^{-1/2},
\end{eqnarray*}
where $\eta  \sim \chi^2_{n+d_0-k}$, $\bz_0 \sim \mathcal{N}_p (\mathbf{0} , \bI_p)$, $Q \sim \mathcal{F}(k,n+d_0-2k)$, and $\bU$ uniformly distributed on the unit sphere in $\mathds{R}^k$; moreover, $\eta$, $\bz_0$, $Q$, and $\bU$ are mutually independent.
\end{enumerate}
\end{theorem}

Theorem \ref{stochrep2} provides alternative stochastic representations of the optimal portfolio weights obtained under the diffuse prior and under the conjugate prior. Although more difficult mathematical expressions are present in Theorem \ref{stochrep2}, they are more computationally efficient than the ones provided in Theorem \ref{stochrep}. Namely, there is no need to calculate the inverse of the matrices $\bS_{t,d}^*(\bmu)$ and $\bS_{t,c}^*(\bmu)$ in each simulation run and instead, we only  calculate the inverse of the matrices $\bS_{t,d}$ and $\bS_{t,c}$ once for the whole simulation study. This property surely speeds up the simulation study considerably. Finally, we note that the realizations of the random vector $\bU$, which is uniformly distributed on the unit sphere in $\mathds{R}^k$, are obtained by drawing $\bz$ from the $k$-dimensional standard normal distribution and calculating $\bU=\bz/\sqrt{\bz^\top \bz}$.

The results of Theorem \ref{stochrep2} are used to derive Bayesian estimates for the weights of the multi-period optimal portfolio at the initial period of investment as well as at each time of reallocations. They are presented in Theorem \ref{th3}.

\begin{theorem}\label{th3}
    Under the assumption of Theorem \ref{stochrep}, we get

\begin{enumerate}[(a)]
\item Under the diffuse prior (\ref{noninformativecor1}), the Bayes estimate for the optimal portfolio weights at time point $t$ is given by
\begin{eqnarray*}
\hat{\bw}_{t,d}=\mathbb{E}(\bw_t|\bx_{t,n}) &=& C_t (n-1)\bS_{t,d}^{-1}(\overline{\bx}_{t,d}-r_{f,t+1}\bOne)\,.
\end{eqnarray*}

\item Under the conjugate prior $(\ref{informativenucor1})$ and $(\ref{informativesigmacor1})$, the Bayes estimate for the optimal portfolio weights at time point $t$ is given by
\begin{eqnarray*}
\hat{\bw}_{t,c}=\mathbb{E}(\bw_t|\bx_{t,n}) &=& C_t (n+d_0-k-1)\bS_{t,c}^{-1}(\overline{\bx}_{t,c}-r_{f,t+1}\bOne)\,.
\end{eqnarray*}
\end{enumerate}
\end{theorem}

The proof of the theorem is given in the appendix. It is interesting to note that the estimate for the optimal portfolio weights obtained under the diffuse prior coincides with the expression derived in Section 2.1 for their frequentist estimate since $\bS_{t,d}/(n-1)=\bS_t$.

Finally, we present the expressions for the covariance matrices of the optimal portfolio weights in Theorem \ref{th4} with the proof moved to the appendix. These formulas characterize the dependencies between the portfolio weight and also allow to access their Bayesian risk.

\begin{theorem}\label{th4}
    Under the assumption of Theorem \ref{stochrep}, we get:
\begin{enumerate}[(a)]
\item Under the diffuse prior (\ref{noninformativecor1}), the covariance matrix  of $\bw_t$ is given by
\begin{eqnarray*}
\mathbf{V}_{t,d}=\mathbb{V}ar(\bw_t|\bx_{t,n})&=& C_t^2\Bigg[ (n-1)\bS_{t,d}^{-1} (\overline{\bx}_{t,d}-r_{f,t+1}\bOne)(\overline{\bx}_{t,d}-r_{f,t+1}\bOne)^\top\bS_{t,d}^{-1}\\
&+& \left(\frac{n^2+k-2}{n(n+2)}+\frac{k-1}{k}b_d\right)\bS_{t,d}^{-1} \Bigg],
 \end{eqnarray*}
where $b_d=n (\overline{\bx}_{t,d}-r_{f,t+1}\bOne)^\top \bS_{t,d}^{-1} (\overline{\bx}_{t,d}-r_{f,t+1}\bOne)$.

\item Under the conjugate prior (\ref{informativenucor1}) and (\ref{informativesigmacor1}), the covariance matrix  of $\bw_t$ is given by
\begin{eqnarray*}
\mathbf{V}_{t,c}=\mathbb{V}ar(\bw_t|\bx_{t,n})&=& C_t^2\Bigg[ (n+d_0-k-1)\bS_{t,c}^{-1} (\overline{\bx}_{t,c}-r_{f,t+1}\bOne)(\overline{\bx}_{t,c}-r_{f,t+1}\bOne)^\top\bS_{t,c}^{-1}\\
&+& \left(\frac{(n+d_0-k)^2+k-2}{(n+r_0)(n+d_0-k+2)}+\frac{(n+d_0-k)(k-1)}{(n+r_0)k}b_c\right)\bS_{t,c}^{-1} \Bigg],
 \end{eqnarray*}
where $b_c=(n+r_0) (\overline{\bx}_{t,c}-r_{f,t+1}\bOne)^\top \bS_{t,c}^{-1} (\overline{\bx}_{t,c}-r_{f,t+1}\bOne)$.
\end{enumerate}\end{theorem}

The results of Theorems \ref{th3} and \ref{th4} provide the first two moments of optimal portfolio weights and, consequently, they characterize their mean values, variances, and correlations. Although different formulas are obtained under the diffuse prior and under the conjugate prior, when the sample size increases the difference between the corresponding expressions becomes negligible.


More general results are provided in Theorem \ref{th5} where it is shown that $\bw_t$ converge to the same asymptotic normal distribution under the diffuse prior and under the conjugate prior.

\begin{theorem}\label{th5}
	Under the assumption of Theorem \ref{stochrep}, it holds that
\begin{eqnarray*}
\sqrt{n}(\bw_t-\hat{\bw}_t)|\bx_{t,n} &\stackrel{d}{\longrightarrow}& \mathcal{N}\Bigg(0,
C_t^2\Bigg[\breve{\bS}_{t}^{-1} (\breve{\bx}_{t}-r_{f,t+1}\bOne)(\breve{\bx}_{t}-r_{f,t+1}\bOne)^\top\breve{\bS}_{t}^{-1}\\
&+& \left(1+\frac{k-1}{k} (\breve{\bx}_{t}-r_{f,t+1}\bOne)^\top \breve{\bS}_{t}^{-1} (\breve{\bx}_{t}-r_{f,t+1}\bOne)  \right)\breve{\bS}_{t}^{-1}\Bigg]\Bigg)
\end{eqnarray*}
as $n \longrightarrow \infty$ under both the diffuse prior and the conjugate prior where
\[ \breve{\bx}_t\equiv \lim_{n \longrightarrow \infty} \overline{\bx}_{t,d}=\lim_{n \longrightarrow \infty} \overline{\bx}_{t,c}
 ~~\text{and}~~
 \breve{\bS}_t\equiv \lim_{n \longrightarrow \infty} \frac{\bS_{t,d}}{n-1}=\lim_{n \longrightarrow \infty} \frac{\bS_{t,c}}{n+r_0}\]
and
\[\hat{\bw}_t\equiv \lim_{n \longrightarrow \infty} \hat{\bw}_{t,d} = \lim_{n \longrightarrow \infty} \hat{\bw}_{t,c}= C_t \breve{\bS}_{t}^{-1}(\breve{\bx}_{t}-r_{f,t+1}\bOne).\]

\end{theorem}

The proof of Theorem \ref{th5} is given in the appendix. Its results are in line with the Bernstein-von Mises theorem (c.f., \cite{BernardoSmith2000}) which shows under some regularity conditions that the posterior distribution converges to the normal one independently of the prior used when the sample size tends to infinity. In practice, the asymptotic covariance matrix of $\bw_t$ is approximated by using $\overline{\bx}_t$ and $\bS_t$ instead of $\breve{\bx}_t$ and $\breve{\bS}_t$.

\subsection{Posterior predictive distribution}
In this section we derive the posterior predictive distribution of the wealth at time point $t+1$, $\widehat{W}_{t+1}$, given the observable data $\bx_{t,n}$ under the diffuse prior \eqref{noninformativecor1} and the conjugate prior(\ref{informativenucor1}) and (\ref{informativesigmacor1}) for the given vector of portfolio weights $\bv_{t}$ and the current wealth $W_t$. Namely, the aim is to derive the posterior predictive distribution of
\begin{equation}\label{wealth_t+1}
  W_{t+1}=W_{t}(1+r_{f,t} + \bv_{t}^\top (\bX_{t+1}-r_{f,t+1}))
\end{equation}
given information provided by the observation matrix $\bx_{t,n}$, i.e.
\begin{eqnarray*}
  f_{\hat{W}_{t+1}}(w|\bx_{t,n}) &=& \int_{\bmu,\bSigma} f_{\hat{W}_{t+1}}(w|\bmu,\bSigma,\bx_{t,n}) \pi(\bmu,\bSigma|\bx_{t,n})\mbox{d}\bmu \mbox{d}\bSigma,
\end{eqnarray*}
where $\pi(\bmu,\bSigma|\bx_{t,n})$ is the posterior distribution obtained under the diffuse prior or the conjugate prior. The symbol $\hat{W}_{t+1}$ denotes a random variable whose distribution coincides with the posterior predictive distribution of the wealth calculated at time point $t+1$.

In Theorem \ref{th6} we present the stochastic representations of the posterior predictive distribution of $\hat{W}_{t+1}$ with the proof given in the appendix. The symbol $t_{d}$ stands for the standard univariate $t$-distribution with $d$ degrees of freedom.

\begin{theorem}\label{th6}
Under the assumption of Theorem \ref{stochrep} we get:
\begin{enumerate}[(a)]
\item Under the diffuse prior (\ref{noninformativecor1}), the stochastic representation of the posterior predictive distribution of $W_{t+1}$ is given by
\begin{eqnarray*}
\widehat{W}_{t+1} &\stackrel{d}{=}& W_t\Bigg(1+r_{f,t+1} + \bv_{t}^\top (\overline{\bx}_{t,d}-r_{f,t+1})\\
&+&\sqrt{\bv_{t}^\top \bS_{t,d} \bv_t}\left(\frac{t_1}{\sqrt{n(n-k)}}+ \sqrt{1+\frac{t_1^2}{n-k}}\frac{t_2}{\sqrt{n-k+1}}\right)\Bigg)
\end{eqnarray*}
where $t_1$ and $t_2$ are independent with $t_1\sim t_{n-k}$ and $t_2\sim t_{n-k+1}$.

\item Under the conjugate prior $(\ref{informativenucor1})$ and $(\ref{informativesigmacor1})$, the stochastic representation of the posterior predictive distribution of $W_{t+1}$ is given by
\begin{eqnarray*}
\widehat{W}_{t+1} &\stackrel{d}{=}& W_t\Bigg(1+r_{f,t+1} + \bv_{t}^\top (\overline{\bx}_{t,c}-r_{f,t+1})\\
&+&\sqrt{\bv_{t}^\top \bS_{t,c} \bv_t}\left(\frac{t_1}{\sqrt{(n+r_0)(n+d_0-2k)}}+ \sqrt{1+\frac{t_1^2}{n+d_0-2k}}\frac{t_2}{\sqrt{n+d_0-2k+1}}\right)\Bigg)\,,
\end{eqnarray*}
where $t_1$ and $t_2$ are independent with $t_1\sim t_{n+d_0-2k}$ and $t_2\sim t_{n+d_0-2k+1}$.
\end{enumerate}
\end{theorem}

The results in Theorem \ref{th6} are very useful in analyzing the behavior of the investor's wealth during the whole investment period as well as at the final point $T$. It allows: (i) to calculate with which probability the investor can become bankrupt during the whole investment horizon at each time point; (ii) to construct the prediction intervals for the wealths at each time point of the investment period; (iii) to determine risk measures, like Value-at-Risk (VaR) and conditional VaR (CVaR), of the investment strategy during all times of the future reallocation; (iv) to specify a region where the final wealth belongs to with a high probability. We illustrate these results based on real data in Section 3.

\section{Empirical study}
\subsection{Data description}

The data used in the empirical study consist of weekly returns on twelve stocks from the FTSE 100, namely Barclays, Glaxo Smith Kline, Standard Life,  Marks and Spencer, Burberry Group plc, HSBC, LLoyds Banking, NEXT plc, Rolls-Royce Holding, The Sage Group, Tesco plc and Unilever which represent a variety of branches with strong international activities.  Since the parameters of the asset returns are not usually constant over a longer period of time, we disregard the use of monthly data which are closer to the normal distribution and choose weekly returns as a compromise between actuality and the assumption of conditional normality. As a risk-free rate we use the weekly returns on the three-months US treasury bill.

\begin{figure}[t]
	\centering
	\includegraphics[width=15cm]{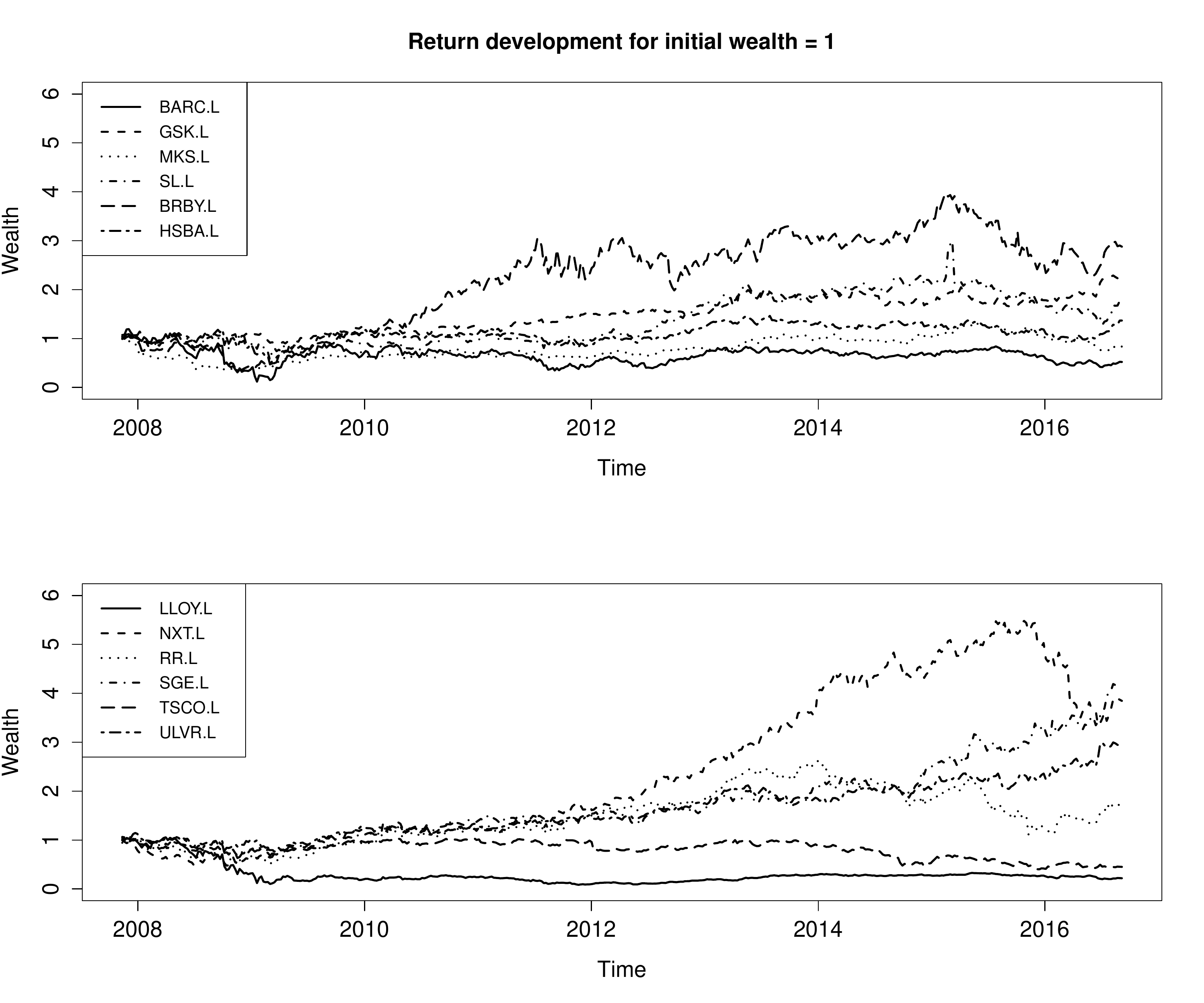}
	\caption{\footnotesize Development of the gross returns for the twelve assets considered in the portfolio.\label{Returns}}
\end{figure}

The portfolio weights are estimated using a rolling window estimation with different sample sizes of $n$ $\in$ $\{52,78, 104,130\}$ corresponding to one year up to two and a half years of weekly data in steps of six months. The portfolio runs from 6.6.2016 until 5.9.2016 ($T=13$) covering a precarious market situation due to Great Britains referendum to leave the European Union on 23.06.2016. The gross returns of these assets are given in Figure \ref{Returns}. Especially Barclays suffered a loss of nearly 10 $\%$ in the week after the Brexit decision but also suffered losses in the weeks prior to the Brexit. HSBC announced that significant parts of her banking operations is moved from the City of London to different locations as a direct reaction to the referendum and it is rumoured that Lloyds seeks for a German banking licence as a consequence to the Brexit. The returns of the Marks and Spencer share were not as affected by the Brexit but the company reported that consumer confidence would be weakened in the days prior to the Brexit. This also implies price uncertainty for domestic consumer products due to a decline of the pound losing almost a fifth of his value against the dollar after the Brexit vote, which was emphasized for example by Tesco and Unilever. But Glaxo Smith Kline and Standard Life seem to be unaffected by the Brexit decision, yielding even positive returns. Rolls Royce, after all, faced significant losses in the beginning of 2016 and is hit by the Brexit vote severely, since they need to hedge a huge amount of British pounds against currency fluctuations because most of the contracts in aerospace are conducted in dollars.

\subsection{Posterior distribution of the weights}

Due to Theorem \ref{stochrep2} it is possible to access the posterior distribution of the weights directly. The weights can be sampled using the following procedure:
\begin{enumerate}
	\item Generate independently
	\begin{itemize}
		\item $\eta$ $\sim$ $\chi_n ^2$ under the diffuse prior or $\eta$ $\sim$ $\chi_{n+d_0-k} ^2$ under the conjugate prior
		\item $\bz_0$ $\sim$ $\mathcal{N}_p (\bNull,\bI_p)$
		\item $Q$ $\sim$ $\mathcal{F}(k,n-k)$ under the diffuse prior or $Q$ $\sim$ $\mathcal{F}(k,n+d_0-2k)$ under the conjugate prior
		\item $\bZ$ $\sim$ $\mathcal{N}_k(\bNull, \bI_k)$ $\rightarrowtail$ $\bU = \bZ/\sqrt{\bZ^\prime\bZ}$
	\end{itemize}
	\item Compute the vector of portfolio weights by using the stochastic representation (\ref{th2_eq1}) for the diffuse prior or (\ref{th2_eq2}) for the conjugate prior.
	\item Repeat steps (1) and (2) $B$ times.
\end{enumerate}

The implementation of this simulation procedure leads to sequences of optimal portfolio weights of size B at each time point of the investment period, from which using their sample distribution we approximate the posterior distributions of the weights as well as their important quantiles from these distributions and the credible sets for portfolio weights. It is remarkable that all computations can easily be done by generating samples from the well known univariate distributions and high numerical precision could be achieved by choosing the corresponding value of $B$.

\begin{figure}[t]
	\centering
	\includegraphics[width=17cm]{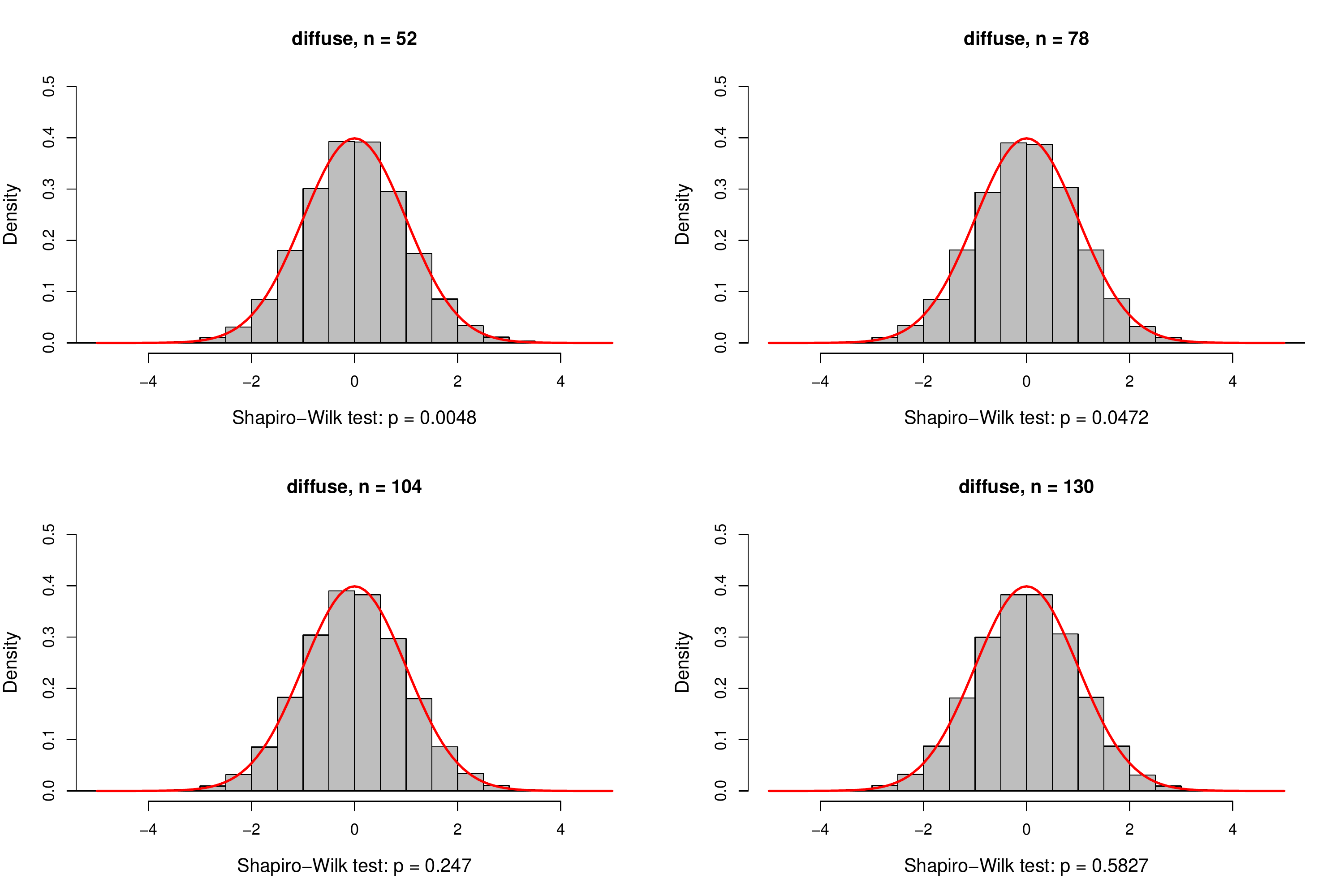}
	\caption{\footnotesize Histograms of the standardized Glaxo Smith Kline (GSK) weight for the diffuse prior. The hypothesis that the weight is normally distributed can not be rejected for common significance levels when the sample size is larger than $n=100$. \label{dhist}}
\end{figure}

\begin{figure}[t]
	\centering
	\includegraphics[width=17cm]{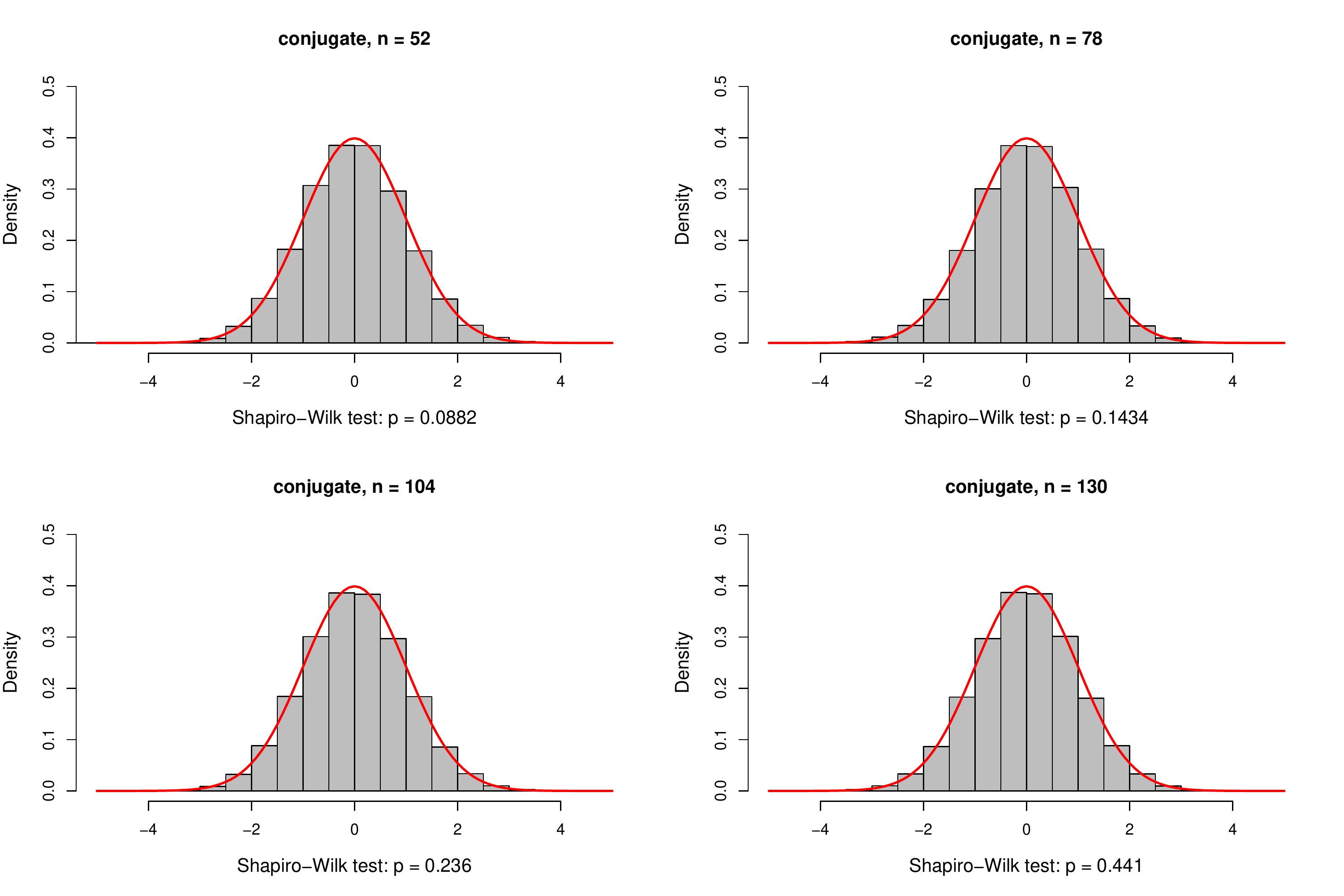}
	\caption{\footnotesize Histograms of the standardized Glaxo Smith Kline (GSK) weight for the conjugate prior. The hypothesis that the weight is normally distributed can not be rejected for common significance levels in the case of all considered sample sizes. \label{chist}}
\end{figure}

In Figures \ref{dhist} and \ref{chist}, we analyze the finite-sample behavior of the results presented in Theorem \ref{th5}. Namely, we investigate the speed of convergence of the posterior distribution of the optimal portfolio weights to the corresponding asymptotic distribution which is a normal distribution according to Theorem \ref{th5} for both priors. The choice of the hyperparameters $\bm_0$ and $\bS_0$ in the case of the conjugate prior are of particular interest. According to the Bayesian paradigm, $\bm_0$ and $\bS_0$ represent the correct belief of the decision maker. In practice, however, there are several data driven methods how to replace $\bm_0$ and $\bS_0$ by data-dependent values $\hat{\bm}_0$ and $\hat{\bS}_0$. We make use of the empirical Bayes approach (see Section 5.2 in the appendix for the derivation of the formulas) which is applied to the weekly data of the returns on the corresponding assets directly from the time period before the empirical counterparts of the portfolio weights are estimated, always with the same time window. Namely, they are given by
\begin{eqnarray*}
	\hat{\bm}_0 &=& \overline{\bx}_{n-t} ~~ \text{and} ~~ 	\hat{\bS}_0 = \frac{(d_0 - k - 1)(n-1)}{n} \bS_{n-t}
\end{eqnarray*}
with the derivation moved to the appendix (Section 5.2). The prior parameters for $t>1$ are estimated using a rolling window starting in the corresponding period. We set $d_0$ equal to the number of observations in the pres-sample period, i.e., $d_0=n$.

We set $B=10^5$ for draws from the stochastic representations of Theorem \ref{stochrep2} and compare the standardized weight of Glaxo Smith Kline (GSK) calculated for the priod $T-1$ in the case of several sample sizes $n$ $\in$ $\{52,78, 104,130\}$. The corresponding histograms are given in Figure \ref{dhist} for the diffuse prior and in Figure \ref{chist} for the conjugate prior. In both figures we also present the p-values of the Shapiro-Wilk test, indicating if the standardized weights follow a standard normal distribution. This hypothesis is rejected for $n=52$ and $n=78$ in the case of the diffuse prior for a common significance level of 5 $\%$ but it cannot be rejected at this level for larger sample sizes. Stronger results are obtained in the case of the conjugate prior, where the null hypothesis cannot be rejected at 5 $\%$ level for all considered sample sizes. We therefore conclude that the approximate distribution of Theorem \ref{th5} works reasonably well.

\subsection{Wealth development and credibility intervals}

Since the main purpose of investing is making money, investors are therefore interested in how much money they made during an investment period. We focus again on the same investment period covering the Brexit-referendum as in the previous subsection.

\begin{figure}[t]
	\centering
	\includegraphics[width=17cm]{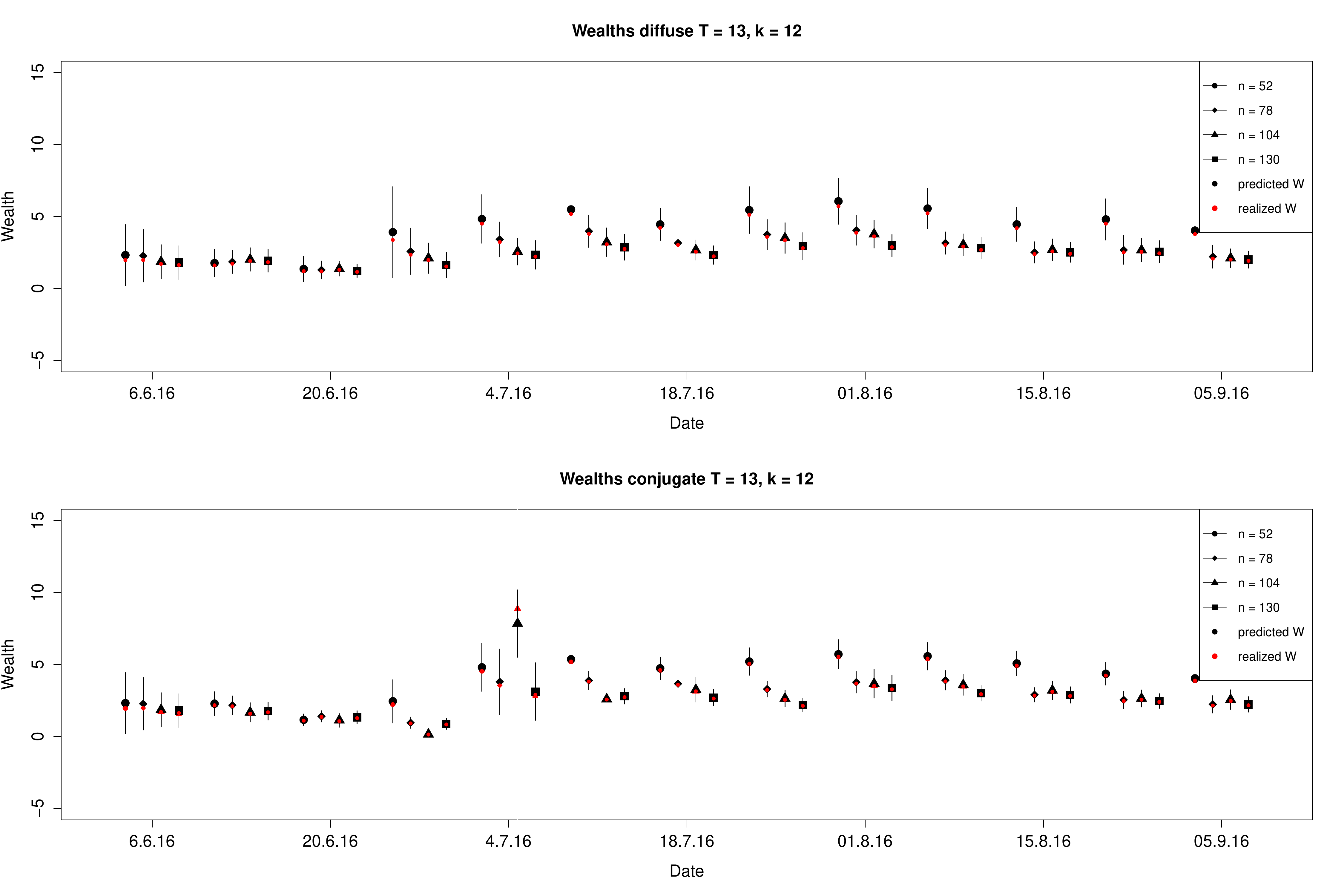}
	\caption{\footnotesize Wealth development and $95\%$ credible intervals for the diffuse prior (above) and for the conjugate prior (below). The wealth for smaller $n$ is almost always higher compared to a portfolio estimated with larger $n$, while the credible intervals are much narrower for larger $n$. \label{wealth_d}}
\end{figure}

During the lifetime of the portfolio, no bankruptcy occurred. But more importantly, the stochastic representation for the posterior predictive distribution given in Theorem 6 can be used to calculate credible intervals for the wealth. By generating $B=10^5$ draws from Theorem 6 and calculating the 95 $\%$ credible intervals, we generate upper and lower bounds for the wealth in the specific period. These intervals together with the predicted and realized wealths are shown in Figure \ref{wealth_d}. We observe a difference in the width of the intervals for lower and larger sample sizes which was expected. The credible intervals are considerably smaller for $n$ $\in$ $\{104, 130\}$ compared to smaller $n$. Note that the sample size has to be sufficiently large in relation to the number of assets. Otherwise, the credible intervals are inflated due to massive estimation uncertainty known as the curse of dimensionality.

It might happen that both the diffuse and the conjugate priors do not perform well when the sample size increases. The reason for the diffuse prior is that the empirical counterparts might not describe the portfolio running period well, indicating a trade-off between the actuality and stability of the parameters. This problem is amplified for the conjugate prior since the prior parameters are determined using even more distant data. While the data-driven approach to the conjugate prior is somewhat realistic, it is not completely in line with the Bayesian paradigm. When the expectations and therefore the choice of hyperparameters are closer to the return behaviour after the Brexit, the results could be improved. Although this is consistent with the Bayesian paradigm, such an approach is of course not entirely practical but not impractical: using appropriate forecasting methods, other data driven methods can be applicable as long as they yield a reliable point estimate. This subjective approach emphasizes the possibility as well as the necessity to resemble realistic future market behaviour in the prior parameterization and it is left for future research.

\subsection{Default probability}
Due to the accessability of the posterior predictive distribution, we can also calculate the default probability of our portfolio at each time point, defined as the event that our wealth becomes negative at this point in time. The predictive probability of default can easily be determined by calculating the amount of defaults in relation to all draws, in this case $B=10^5$. The development of the defaults is given in Figure \ref{PD}. Again, we find a pattern resembling the credible intervals of the posterior predictive distribution illustrated in the previous section with no surprises.

\begin{figure}[t]
	\centering
	\includegraphics[width=17cm]{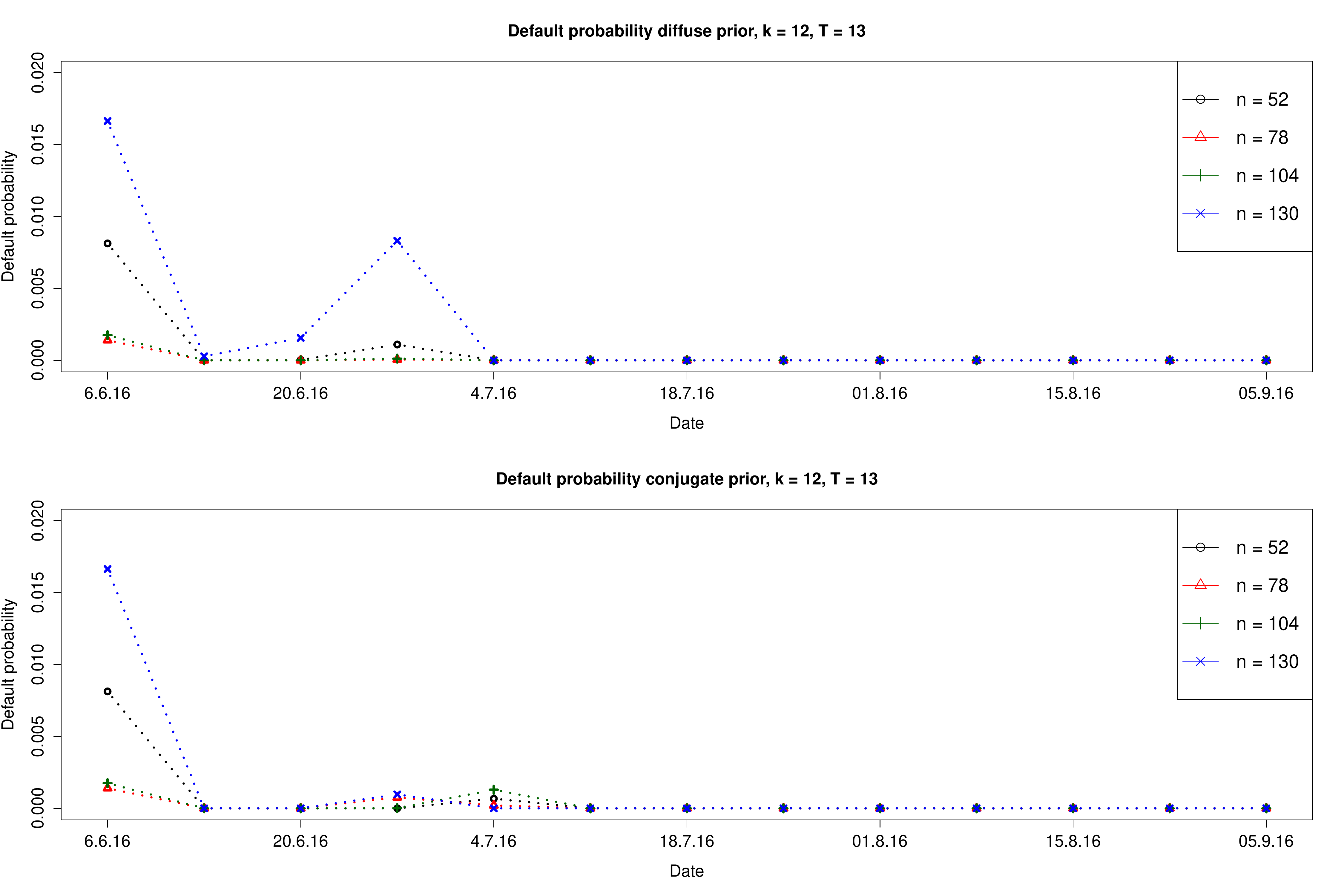}
	\caption{\footnotesize Default probabilities for the diffuse prior (above) and for the conjugate prior (below).\label{PD}}
\end{figure}

Starting with the diffuse prior, we observe a slightly increased default probability on 27.6.2016, the week after the Brexit referendum. With the conjugate prior, this default probability is lower in the same week. Again, the peak for $n=130$ of the diffuse prior again resembles the trade-off between parameter stability and actuality, resulting here in a slightly increased default probability. The default probability for the conjugate prior is slightly increased in the following week compared to the diffuse prior, presumably due to parameters relying on a wider estimation window.

\section{Summary}
In this paper we consider the estimation of the multi-period portfolio for an exponential utility function in a Bayesian setting. Since the portfolio weights are given as the product of two multivariate/matrix-variate random quantities, accessing the distribution of the weights is a challenging task. By choosing the non-informative and the conjugate prior, the posterior distributions of the weights have pleasing properties since the conditional distribution of the precision matrix for a given return vector is an inverted Wishart distribution. With this insight we could use this well understood distribution (c.f. \cite{Muirhead1982}) to derive stochastic representations for the weights which is a direct access to the posterior distribution. Furthermore, these representations also provide us with Bayesian estimates for the optimal portfolio weights together with their covariance matrix. In addition to this, we derive the posterior predictive distribution for the wealth which makes it possible to calculate the quantiles of the portfolio wealth at each time point of the investment period and it is therefore highly relevant for risk purposes. The method is then applied to real data from the FTSE 100 covering the period of the Brexit referendum. With these data we determine the posterior distribution of the weights, the predictive wealths in each period, the lower wealth quantiles as well as the default probability in every time period.

It turns out that the use of stochastic representations to generate the posterior distribution numerically is computationally highly efficient: the representations rely on samples from well known distributions and no MCMC methods are needed. In the empirical part of Section 3 it was demonstrated that these methods work well and are easy to implement. We have to emphasize several points: while the non-informative prior will yield results which coincide with the common frequentist case and is as easily to apply as the classical case, the conjugate or informative prior is said to involve a potentially large degree of subjectivity -- sometimes implying that the frequentist approach or the non-informative prior would be objective. But we have to choose the sample size in all of these cases which is naturally a subjective choice with a huge effect on the performance of the portfolio as we demonstrate in Section 3. This trade-off between parameter actuality and parameter stability has to be faced by the practitioner. One advantage of the conjugate prior is of course that we can incorporate our beliefs regarding the future behaviour of the asset returns in our model which is not possible neither in the frequentist nor in the non-informative case. This is clearly at the core of every investment decision and reflects natural decision making. Nevertheless, the hyperparameters have to be chosen carefully and a rigorous sensitivity analysis is left for future research.

There are still other open research questions regarding the multi-period portfolio choice with exponential utility function which are left for future research. The present approach can be extended to the case with predictable variables as discussed in \cite{bodnar2015exact} in the case of the known parameters of the asset return distribution. This, however, is much more difficult due to the more complicated structure of the optimal portfolio weights and the dependence structure of the asset returns. Furthermore, the multi-period optimal portfolios obtained by using other utility functions can be estimated following the approach suggested in the paper.

\section{Appendix}

\subsection{Proofs of the theorems}
In this part of the paper we present the proofs of the theoretical results. First, we note that the derived posterior distributions under the diffuse prior and under the conjugate prior in Proposition \ref{prop2} have a similar structure. For that reason, we formulate and prove some lemmas from which the results in both cases of the diffuse prior and the conjugate prior follow.

\begin{lemma}\label{lem1}
Let
\begin{eqnarray*}
\bOmega | \bnu, \by \sim \mathcal{IW}_k (k_y, \bS_y^*(\bnu)) ~~ \text{and} ~~ \bnu | \by \sim  t_k\left(d_y,\bm_y,\bS_y/d_y\right),
\end{eqnarray*}
where $\bS_y^*(\bnu)=v_y(\bS_y+(\bnu-\bm_y)(\bnu-\bm_y)^\top)$ and let $\bM$ be a $p\times k$-dimensional matrix of constants. Then the stochastic representation of
$\bM \bOmega^{-1} (\bnu-\mathbf{a})$ is given by
\begin{eqnarray*}
&&\bM \bOmega^{-1} (\bnu-\mathbf{a})\stackrel{d}{=}\eta \bM \bS_y^*(\bnu)^{-1} (\bnu-\mathbf{a} )\\
&+&\sqrt{\eta}\left((\bnu-\mathbf{a})^\top \bS_y^*(\bnu)^{-1} (\bnu-\mathbf{a} )\cdot\bM \bS_y^*(\bnu)^{-1} \bM^\top-\bM \bS_y^*(\bnu)^{-1} (\bnu-\mathbf{a} )(\bnu-\mathbf{a})^\top \bS_y^*(\bnu)^{-1} \bM^\top\right)^{1/2}\bz_0,
\end{eqnarray*}
where $\eta  \sim \chi^2_{k_y - k - 1}$, $\bz_0 \sim \mathcal N _p (\mathbf 0 , \mathbf I_p)$, and $\bnu | \by \sim t_k\left(d_y,\bm_y,\bS_y/d_y\right)$; moreover, $\eta, \mathbf z_0$ and $\bnu$ are mutually independent.
\end{lemma}

\begin{proof}[Proof of Lemma \ref{lem1}]
Since $\bOmega^*\stackrel{d}{=}\bOmega | \bnu=\bnu^*, \by \sim \mathcal{IW}_k (k_y, \bS_y^*(\bnu^*))$ and, consequently, $\bOmega^{*\,-1}\sim \mathcal{W}_k (k_y-k-1, \bS_y^*(\bnu^*)^{-1})$ (c.f., Theorem 3.4.1 in \cite{GuptaNagar2000}, it holds that (see, e.g., Theorem 3.2.5 in \cite{Muirhead1982})
\begin{eqnarray*}
\bXi^*=\widetilde {\bM} \bOmega^{*\,-1} \widetilde {\bM}^\top \sim \mathcal{W}_k (k_y-k-1, \bV^*),
\end{eqnarray*}
with $\widetilde {\bM} = (\bM^\top, \bnu^*-\mathbf{a})^\top$ and $\bV^*=\tilde {\bM} \bS_y^*(\bnu^*)^{-1} \tilde {\bM}^\top$. Next, we partition $\bXi^*$ and $\bV^*$ in the following way
\[\bXi^*=\left(
           \begin{array}{cc}
             \bXi^*_{11} & \bXi^*_{12} \\
             \bXi^*_{21} &  \Xi^*_{22} \\
           \end{array}
         \right)
=\left(
   \begin{array}{cc}
     \bM \bOmega^{*\,-1} \bM^\top & (\bnu^*-\mathbf{a})^{\top} \bOmega^{*\,-1} \bM^\top \\
     \bM \bOmega^{*\,-1} (\bnu^*-\mathbf{a}) & (\bnu^*-\mathbf{a})^{\top} \bOmega^{*\,-1} (\bnu^*-\mathbf{a}) \\
   \end{array}
 \right)\]
and
\[\bV^*=\left(
           \begin{array}{cc}
             \bV^*_{11} & \bV^*_{12} \\
             \bV^*_{21} & V^*_{22} \\
           \end{array}
         \right)
=\left(
   \begin{array}{cc}
     \bM \bS_y^*(\bnu^*)^{-1} \bM^\top & (\bnu^*-\mathbf{a})^{\top} \bS_y^*(\bnu^*)^{-1} \bM^\top \\
     \bM \bS_y^*(\bnu^*)^{-1} (\bnu^*-\mathbf{a})& (\bnu^*-\mathbf{a})^{\top} \bS_y^*(\bnu^*)^{-1} (\bnu^*-\mathbf{a}) \\
   \end{array}
 \right)\,.\]
The application of Theorem 3.2.10 in \cite{Muirhead1982} yields
\[\bXi_{12}^* | \Xi_{22}^* \sim \mathcal{N}_p (  \bV_{12}    V_{22}^{-1} \Xi_{22}^*,   \bV_{11 \cdot 2}\Xi_{22}^*) ~~ \text{with} ~~
\bV_{11 \cdot 2} = \bV_{11} - \frac{\bV_{12}\bV_{21}}{V_{22}}.\]

Defining $\eta = \Xi_{22}^* / V_{22}$ and using Theorem 3.2.8 of \cite{Muirhead1982} we get that $\eta \sim \chi^2_{k_y - k-1}$. Since the $\chi^2_{k_y - k-1}$-distribution is independent of $\bnu=\bnu^*$ and $\by$ (on which the distribution of $\Xi_{22}^*$ depends on by definition of $\bXi^*$), it is also the unconditional distribution of $\eta$ as well as $\eta$ is independent of both $\bnu$ and $\by$. Thus, the stochastic representation of  $\bM \bOmega^{-1} (\bnu-\mathbf{a})$ is given by
\begin{eqnarray*}
\bM \bOmega^{-1} (\bnu-\mathbf{a} )&\stackrel{d}{=}&\eta \bM \bS_y^*(\bnu)^{-1} (\bnu-\mathbf{a} )+\sqrt{\eta}\Bigg((\bnu-\mathbf{a} )^\top \bS_y^*(\bnu)^{-1} (\bnu-\mathbf{a} )\cdot\bM \bS_y^*(\bnu)^{-1} \bM^\top\\
&-&\bM \bS_y^*(\bnu)^{-1} (\bnu-\mathbf{a} )(\bnu-\mathbf{a})^\top \bS_y^*(\bnu)^{-1} \bM^\top\Bigg)^{1/2}\bz_0,
\end{eqnarray*}
where $\eta \sim \chi^2_{k_y- k - 1}$, $\bz_0 \sim \mathcal{N} _p (\mathbf{0}, \bI_p)$, and $\bnu | \by \sim t_k\left(d_y,\bm_y,\bS_y/d_y\right)$; moreover, $\eta$, $\bz_0$ and $\bnu$ are mutually independent. This completes the proof of the lemma.
\end{proof}

\begin{proof}[Proof of Theorem \ref{stochrep}]
The results of Theorem \ref{stochrep} follow from Lemma \ref{lem1} with $\bM=C_t\bL$, $\bSigma=\bOmega$, $\bnu=\bmu$, $\mathbf{a}=r_{f,t+1}\bOne$ and
\begin{enumerate}[(a)]
\item $k_y=n+k+1$, $d_y=n-k$, $v_y=n$, $\bm_y=\overline{\bx}_{t,d}$, $\bS_y=\bS_{t,d}/n$, and $\bS_y^*(\bnu)=\bS_{t,d}^*(\bmu)$ in the case of the diffuse prior;
\item $k_y=n+d_0+1$, $d_y=n+d_0-2k$, $v_y=n+r_0$, $\bm_y=\overline{\bx}_{t,c}$, $\bS_y=\bS_{t,c}/(n+r_0)$, and $\bS_y^*(\bnu)=\bS_{t,c}^*(\bmu)$  in the case of the conjugate prior.
\end{enumerate}
\end{proof}

\begin{lemma}\label{lem2}
Under the conditions of Lemma 1, we get the following stochastic representation of $\bM \bOmega^{-1} (\bnu-\mathbf{a})$ expressed as
\begin{eqnarray*}
\bM \bOmega^{-1} (\bnu-\mathbf{a}) \stackrel{d}{=} v_y^{-1}\eta \bM \bzeta +
v_y^{-1}\sqrt{\eta}\left(\epsilon \bM \bUpsilon \bM^\top -\bM \bzeta \bzeta^\top \bM^\top\right)^{1/2} \bz_0,
 \end{eqnarray*}
with
\begin{eqnarray*}
\epsilon &=& \epsilon (Q, \bU) = (\bm_y-\mathbf{a})^\top \bS_y ^{-1} (\bm_y-\mathbf{a})
+ 2\frac{ \sqrt{k Q/d_y}}{1+k Q/d_y}(\bm_y-\mathbf{a})^\top \bS_y^{-1/2} \bU \\
&+& \frac{k Q/d_y}{1+ k Q/d_y} -\frac{k Q/d_y}{1+k Q/d_y}\left( (\bm_y-\mathbf{a})^\top \bS_y^{-1/2} \bU\right)^2,
\\
\bzeta&=&\bzeta (Q, \bU)= \bS_y^{-1} (\bm_y-\mathbf{a})+\frac{ \sqrt{k  Q/d_y} }  {   1+ k Q/d_y}\bS_y^{-1/2} \bU
-\frac{k Q/d_y }  {   1+ k Q/d_y} \bS_y^{-1/2} \bU\bU^\top \bS_y^{-1/2} (\bm_y-\mathbf{a}),
\\
\bUpsilon &=&\bUpsilon (Q, \bU)=\bS_y^{-1}-\frac{k Q/d_y }  {   1+ k Q/d_y} \bS_y^{-1/2} \bU\bU^\top \bS_y^{-1/2},
\end{eqnarray*}
where $\eta  \sim \chi^2_{k_y-k-1}$, $\bz_0 \sim \mathcal{N}_p (\mathbf{0} , \bI_p)$, $Q \sim \mathcal{F}(k,d_y)$, and $\bU$ uniformly distributed on the unit sphere in $R^k$; moreover, $\eta$, $\bz_0$, $Q$, and $\bU$ are mutually independent.
\end{lemma}

\begin{proof}[Proof of Lemma \ref{lem2}]
The application of the Sherman-Morrison formula (see, e.g., p.125 in \cite{Meyer2000}) yields
\begin{equation}\label{lem2eq1}
(\bS_y+(\bnu-\bm_y)(\bnu-\bm_y)^\top)^{-1}
= \bS_y^{-1}-\frac{\bS_y^{-1} (\bnu - \bm_y) (\bnu - \bm_y)^\top \bS_y^{-1}} { 1+ (\bnu - \bm_y)^\top \bS_y^{-1}(\bnu - \bm_y)}
\end{equation}

Let
\begin{equation}\label{lem2eq2}
\bU = \frac{ \bS_y^{-1/2} (\bnu - \bm_y)}  {\sqrt{ (\bnu - \bm_y)^\top \bS_y^{-1} (\bnu - \bm_y)} }
~~ \text{and}~~Q=d_y (\bnu - \bm_y) ^\top \bS_y^{-1} (\bnu - \bm_y)/k.
\end{equation}

Since $\bnu | \by \sim t_k(d_y, \bm_y, \bS_y/d_y)$ and that the multivariate $t$-distribution belongs to the class of the elliptically contoured distributions, we obtain that $\bU $ and $Q $ are independent, and $\bU$ is uniformly distributed on the unit sphere in $R^k$ (see Theorem 2.15 of \cite{Gupta2013}).
Moreover, from the properties of the multivariate $t$-distribution (see p. 19 of \cite{Kotz2004}), we get that $Q \sim \mathcal F(k,d_y)$, i.e., $Q$ has an $\mathcal F$-distribution with $k$ and $d_y$ degrees of freedom.

Hence, the application of the \eqref{lem2eq1} and \eqref{lem2eq2} leads to
\begin{eqnarray*}
&&(\bS_y+(\bnu-\bm_y)(\bnu-\bm_y)^\top)^{-1}= \bS_y^{-1} - \frac{ k Q /d_y} {   1+ k Q/d_y} \bS_y^{-1/2} \bU\bU^\top \bS_y^{-1/2},\\
&&(\bS_y+(\bnu-\bm_y)(\bnu-\bm_y)^\top)^{-1} (\bnu-\mathbf{a}) \\
&=&  \bS_y^{-1} (\bnu-\mathbf{a})- \frac{\bS_y^{-1} (\bnu - \bm_y) (\bnu - \bm_y)^\top \bS_y^{-1} (\bnu - \bm_y+ \bm_y-\mathbf{a})}
 { 1+ (\bnu - \bm_y)^\top \bS_y^{-1}(\bnu - \bm_y)}\\
&=& \bS_y^{-1} (\bm_y-\mathbf{a}) +
\frac{\bS_y^{-1} (\bnu - \bm_y)}  { 1+ (\bnu - \bm_y)^\top \bS_y^{-1}(\bnu - \bm_y)}
-\frac{\bS_y^{-1} (\bnu - \bm_y) (\bnu - \bm_y)^\top \bS_y^{-1} (\bm_y-\mathbf{a})}  { 1+ (\bnu - \bm_y)^\top \bS_y^{-1}(\bnu - \bm_y)}\\
&=& \bS_y^{-1} (\bm_y-\mathbf{a})+\frac{ \sqrt{  k Q/d_y} }  {   1+ k Q/d_y}\bS_y^{-1/2} \bU
-\frac{  k Q/d_y }  {   1+   k Q/d_y} \bS_y^{-1/2} \bU\bU^\top \bS_y^{-1/2} (\bm_y-\mathbf{a}),
\end{eqnarray*}
and
\begin{eqnarray*}
&&(\bnu-\mathbf{a})^\top (\bS_y+(\bnu-\bm_y)(\bnu-\bm_y)^\top)^{-1} (\bnu-\mathbf{a})\\
&=& (\bm_y-\mathbf{a})^\top \bS_y ^{-1} (\bm_y-\mathbf{a})
+ 2\frac{(\bm_y-\mathbf{a}) ^\top \bS_y^{-1/2} \bU \sqrt{ k Q/d_y}}{1+ k Q/d_y} \\
&+& \frac{k Q/d_y}{1+ k Q/d_y}
-\frac{k Q/d_y}{1+  k Q/d_y}\left( (\bm_y-\mathbf{a})^\top \bS_y^{-1/2} \bU\right)^2.
\end{eqnarray*}
Putting the above results together we obtain the statement of the lemma.
\end{proof}

\begin{proof}[Proof of Theorem \ref{stochrep2}]
The results of Theorem \ref{stochrep2} are obtained by using Lemma \ref{lem2} with $\bM=C_t\bL$, $\bSigma=\bOmega$, $\bnu=\bmu$, $\mathbf{a}=r_{f,t+1}\bOne$ and
\begin{enumerate}[(a)]
\item $k_y=n+k+1$, $d_y=n-k$, $v_y=n$, $\bm_y-\mathbf{a}=\overline{\bx}_{t,d}-r_{f,t+1}\bOne$, $\bS_y=\bS_{t,d}/n$, and $\bS_y^*(\bnu)=\bS_{t,d}^*(\bmu)$ in the case of the diffuse prior;
\item $k_y=n+d_0+1$, $d_y=n+d_0-2k$, $v_y=n+r_0$, $\bm_y-\mathbf{a}=\overline{\bx}_{t,c}-r_{f,t+1}\bOne$, $\bS_y=\bS_{t,c}/(n+r_0)$, and $\bS_y^*(\bnu)=\bS_{t,c}^*(\bmu)$  in the case of the conjugate prior.
\end{enumerate}
\end{proof}

\begin{proof}[Proof of Theorem \ref{th3}]
The proof of the theorem is based on the stochastic representations obtained in Theorem \ref{stochrep2}. Let $\bl$ be an arbitrary $k$-dimensional vector of constants.
\begin{enumerate}[(a)]
  \item Using that $\eta$, $\bz_0$ $Q$, and $\bU$ are independent and that $\mathbb{E}(\bz_0)=\mathbf{0}$, in the case of the diffuse prior we get
\begin{eqnarray*}
\mathbb{E}(\bl^\top \bw_t|\bx_{t,n}) &=& C_t \mathbb{E}(\eta) \bl^\top \mathbb{E}(\bzeta_d)
 \end{eqnarray*}
with $\mathbb{E}(\eta)=n$ and
\begin{eqnarray*}
\mathbb{E}(\bzeta_d|\bx_{t,n}) &=& \bS_{t,d}^{-1} (\overline{\bx}_{t,d}-r_{f,t+1}\bOne)+ \frac{1}{\sqrt{n}} \mathbb{E}\left(\frac{\sqrt{k  Q/(n-k)} }  { 1+ k Q/(n-k)}\bS_{t,d}^{-1/2}\right) \mathbb{E}(\bU)\\
&-&\mathbb{E}\left(\frac{k Q/(n-k) }  {   1+ k Q/(n-k)} \bS_{t,d}^{-1/2}\right) \mathbb{E}(\bU\bU^\top) \bS_{t,d}^{-1/2} (\overline{\bx}_{t,d}-r_{f,t+1}\bOne)\\
&=&\bS_{t,d}^{-1} (\overline{\bx}_{t,d}-r_{f,t+1}\bOne)-\frac{k}{n}\frac{1}{k}\bS_{t,d}^{-1}(\overline{\bx}_{t,d}-r_{f,t+1}\bOne),
 \end{eqnarray*}
where we use that $E (\bU) = \mathbf{0}$ and $E(\bU\bU^T) = \frac{1}{k} \mathbf{I_k}$ (see, e.g. Gupta et al. (2013)) as well as the fact that if $Q \sim \mathcal F (k, n-k)$, then ${\frac{k}{n-k} Q}/\left( 1 + \frac{k}{n-k} Q\right) \sim Beta \left( \frac{k}{2}, \frac{n-k}{2}\right)$. Hence,
$$E \left( \frac{\frac{k}{(n-k)} Q}{ 1 + \frac{k}{(n-k)} Q} \right) = \frac{k}{n}$$
and, consequently, since $\bl$ was an arbitrary vector, we get
\begin{eqnarray*}
\mathbb{E}(\bw_t|\bx_{t,n}) &=& C_t (n-1)\bS_{t,d}^{-1}(\overline{\bx}_{t,d}-r_{f,t+1}\bOne)\,.
\end{eqnarray*}

\item Similar computations as in part (a) leads to
\[
\mathbb{E}(\bw_t|\bx_{t,n}) = C_t (n+d_0-k-1)\bS_{t,c}^{-1}(\overline{\bx}_{t,c}-r_{f,t+1}\bOne)
\]
under the conjugate prior.
\end{enumerate}
\end{proof}

\begin{lemma}\label{lem3}
Under the assumption of Lemma \ref{lem2} with $\bM=\bb^\top: 1\times k$, we get that
\begin{eqnarray*}
&&v_y^2\mathbb{E} ((\bb^\top \bOmega^{-1} (\bnu-\mathbf{a}))^2 | \by)
=(k_y-k-1)(k_y-k)\Bigg[\left(1-\frac{2}{k+d_y}+\frac{2} {(k+d_y) (k+d_y+2)}\right)c_{12}^2\\
&+&\left(\frac{ d_y}{(k+d_y)(k+d_y+2)}+ \frac{1} {(k+d_y) (k+d_y+2)} c_2\right) c_1\Bigg]\\
&+& (k_y-k-1)\Bigg[\left(\frac{k-1}{k+d_y} +\left(1-\frac{1}{k}-\frac{1}{k+d_y}+\frac{1} {(k+d_y) (k+d_y+2)}\right)c_2\right)c_1\\
&+&\frac{2} {(k+d_y) (k+d_y+2)}c_{12}^2\Bigg],
\end{eqnarray*}
where $c_1  = \bb^\top \bS_y^{-1} \bb$, $c_2  = (\bm_y-\mathbf{a})^\top \bS_y^{-1} (\bm_y-\mathbf{a})$,
and $c_{12} = \bb^\top \bS_y^{-1} (\bm_y-\mathbf{a})$.
\end{lemma}

\begin{proof}
The proof of the lemma is based on the stochastic representations from Lemma \ref{lem2}. Since $\eta$, $\bz_0$, $Q$, and $\bU$ are independent as well as $\mathbb{E}(\bz_0)=\mathbf{0}$ and $\mathbb{E}(\bz_0 \bz_0^\top)=\bI_p$, we obtain
\begin{eqnarray*}
&&v_y^2\mathbb{E}((\bb^\top \bOmega^{-1} (\bnu-\mathbf{a}))^2|\by) = \mathbb{E}(\eta^2) \mathbb{E}((\bb^\top \bzeta)^2|\by) +
\mathbb{E}(\eta)\left(\mathbb{E}(\epsilon \bb^\top \bUpsilon \bb|\by) -\mathbb{E}((\bb^\top \bzeta)^2|\by)\right)\\
&=&(k_y-k-1)(k_y-k)\mathbb{E}((\bb^\top \bzeta)^2|\by)+(k_y-k-1)\mathbb{E}(\epsilon \bb^\top \bUpsilon \bb|\by)
 \end{eqnarray*}
with $\mathbb{E}(\eta)=k_y-k-1$ and $\mathbb{E}(\eta^2)=(k_y-k-1)(k_y-k+1)$.

The application of $E(\bU\bU^T)=\frac{1}{k} \mathbf{I}_k$ and the fact that all odd mixed moments of $\bU$ are zero yield
\begin{eqnarray*}
\mathbb{E}((\bb^\top \bzeta)^2|\by) &=&(\bb^\top\bS_y^{-1} (\bm_y-\mathbf{a}))^2
+ \frac{1}{k} \mathbb{E}\left(\frac{k  Q /d_y}{(1+ k Q/d_y)^2}\right) \bb^\top \bS_y^{-1} \bb\\
&-&\frac{2}{k}\mathbb{E}\left(\frac{ k Q/d_y }{1+ k Q/d_y}\right) (\bb^\top \bS_y^{-1} (\bm_y-\mathbf{a}))^2 \\
&+&E\left(\left( \frac{ k Q/d_y }  {   1+k Q/d_y}\right)^2\right)
E\left((\bb^\top \mathbf S_y^{-1/2} \mathbf {U})^2 ((\bm_y-\mathbf{a})^\top \bS_y^{-1/2} \mathbf {U})^2|\by\right)
\end{eqnarray*}
and
\begin{eqnarray*}
\mathbb{E}\left(\epsilon \bb^\top \bUpsilon \bb|\by\right) &=&
(\bm_y-\mathbf{a})^\top \bS_y ^{-1} (\bm_y-\mathbf{a}) \bb^\top \bS_y^{-1}\bb
 + \mathbb{E}\left(\frac{k Q/d_y}{1+ k Q/d_y} \right) \bb^\top \bS_y^{-1}\bb\\
&-&\frac{1}{k} \mathbb{E}\left(\frac{k Q/d_y}{1+ k Q/d_y}\right) (\bm_y-\mathbf{a})^\top \bS_y^{-1}(\bm_y-\mathbf{a}) \bb^\top \bS_y^{-1}\bb\\
&-&\frac{1}{k}(\bm_y-\mathbf{a})^\top \bS_y^{-1}(\bm_y-\mathbf{a}) \bb^\top \bS_y^{-1}\bb
 - \frac{1}{k}  \mathbb{E}\left(\frac{k Q/d_y}{1+ k Q/d_y}\right) \bb^T \bS_y^{-1}\bb\\
&+&E\left(\left( \frac{  k Q/d_y }  {   1+  k Q/d_y}\right)^2\right) \mathbb{E}\left((\bb^\top\bS_y^{-1/2} \bU)^2 ((\bm_y-\mathbf{a})^\top \bS_y^{-1/2} \bU)^2|\by\right)\,.
\end{eqnarray*}

Since $\frac{k Q/d_y}{ 1+ k Q/d_y}$ has a beta distribution with $k/2$ and $d_y/2$ degrees of freedom, we obtain
\begin{eqnarray*}
E\left( \frac{k Q/d_y}{ 1+ k Q/d_y} \right) &=& \frac{k}{k+d_y},\\
E \left( \frac{k Q/d_y}{ 1+ k Q/d_y} \right)^2 &=& \frac{2k d_y + k^2 (k+d_y +2)} {(k+d_y)^2 (k+d_y+2)}=\frac{k(k+2)} {(k+d_y)(k+d_y+2)}.
\end{eqnarray*}
Furthermore, using $ Q  \sim \mathcal F (k, d_y)$, we get
\begin{eqnarray*}
E \left[ \frac{k Q/d_y}{ (1+ k Q/d_y)^2}\right]
&=&\frac{1}{n_0} \int_0^{\infty}\frac{ k t/d_y}{ (1+ k t/d_y)^2}\frac{1}{ B \left( \frac{k}{2}, \frac{d_y}{2}\right)}\left(\frac{k}{d_y}\right)^{k/2}
t^{k/2-1}\left(1+ \frac{k}{d_y} t\right)^{-(k+d_y)/2} dt\\
&=& \frac{1}{B \left( \frac{k}{2}, \frac{d_y}{2}\right)}\int_0^{\infty}\left(\frac{k}{d_y}\right)^{(k+2)/2}t^{(k+2)/2 -1}
\left(1+ \frac{k}{d_y} t\right)^{-(k+d_y+4)/2}dt\\
&=&\frac{B \left( \frac{k+2}{2}, \frac{d_y +2}{2}\right)}{B \left( \frac{k}{2}, \frac{d_y}{2}\right)}
=\frac{k d_y}{(k+d_y)(k+d_y+2)},
\end{eqnarray*}
where $B (\cdot, \cdot)$ stands for the beta function (see, \citet[p. 256]{Mathai1992}).

Next, we compute $\mathbb{E}\left((\bb^\top\bS_y^{-1/2} \bU)^2 ((\bm_y-\mathbf{a})^\top \bS_y^{-1/2} \bU)^2|\by\right)$. Let $Q_N \sim \chi^2_k$ be independent of $\bU$. Then $\sqrt{Q_N} \bU$ has a multivariate standard normal distribution, i.e.
\begin{eqnarray*}
\begin{pmatrix}
\bb^\top \bS_y^{-1/2}\\
(\bm_y-\mathbf{a})^\top \bS_y^{-1/2}
\end{pmatrix}
\sqrt{Q_N} \bU
&\sim&
\mathcal{N}_{2}
\left(\mathbf 0,\begin{pmatrix}
\bb^\top\bS_y^{-1}\bb&\bb^\top \bS_y^{-1}(\bm_y-\mathbf{a})\\
(\bm_y-\mathbf{a})^\top \bS_y^{-1}\bb&(\bm_y-\mathbf{a})^\top \bS_y^{-1}(\bm_y-\mathbf{a})\\
\end{pmatrix}\right)\\
&=&\mathcal{N}_{2}\left(\mathbf 0,\begin{pmatrix}
c_1&c_{12}\\
c_{12}&c_{2}\\
\end{pmatrix}\right),
\end{eqnarray*}
where $c_1$, $c_2$, and $c_{12}$ are defined in the statement of Lemma \ref{lem3}. Hence,
\begin{eqnarray*}
&&\mathbb{E}\left((\bb^\top \bS_y^{-1/2} \bU)^2 ((\bm_y-\mathbf{a})^\top \bS_y^{-1/2} \bU)^2|\by\right)\\
&=&\mathbb{E}\left[\left(\bb^\top \bS_y^{-1/2} \bU\right)^2  \left(  (\bm_y-\mathbf{a})^\top \bS_y^{-1/2} \bU  \right)^2 | \mathbf y\right]\frac{\mathbb{E} (Q_N^2)}{\mathbb{E} (Q_N^2)}\\
&=&\frac{\mathbb{E}\left[\left(\bb^\top \bS_y^{-1/2} \sqrt{Q_N}\bU\right)^2  \left( (\bm_y-\mathbf{a})^\top \sqrt{Q_N}\bS_y^{-1/2} \bU  \right)^2 | \by\right]
}{\mathbb{E} (Q_N^2)}=\frac{c_1 c_2 + 2 c_{12}^2}{k (k+2)},
\end{eqnarray*}
where the last equality follows from the Isserlis' theorem (c.f., \cite{Isserlis1918}).

Hence,
\begin{eqnarray*}
E (\bb^\top \bzeta \bzeta^\top \bb) &=& c_{12}^2+\frac{1}{k}\frac{k d_y}{(k+d_y)(k+d_y+2)} c_1\\
&-&\frac{2}{k}\frac{k}{k+d_y} c_{12}^2 +\frac{k (k+2)} {(k+d_y) (k+d_y+2)} \frac{c_1 c_2 + 2 c_{12}^2}{k (k+2)}\\
&=&\left(1-\frac{2}{k+d_y}+\frac{2} {(k+d_y) (k+d_y+2)}\right)c_{12}^2\\
&+&\left(\frac{ d_y}{(k+d_y)(k+d_y+2)}+ \frac{1} {(k+d_y) (k+d_y+2)} c_2\right) c_1
\end{eqnarray*}
and
\begin{eqnarray*}
&&E\left(\epsilon \bb^\top \bUpsilon \bb\right) =
c_1c_2 +  \frac{k}{k+d_y} c_1-\frac{1}{k} \frac{k}{k+d_y} c_1c_2\\
&-&\frac{1}{k}c_1c_2- \frac{1}{k}  \frac{k}{k+d_y}c_1+\frac{k (k+2)} {(k+d_y) (k+d_y+2)} \frac{c_1 c_2 + 2 c_{12}^2}{k (k+2)}\\
&=&\frac{2} {(k+d_y) (k+d_y+2)}c_{12}^2+\left(\frac{k-1}{(k+d_y)} +\left(1-\frac{1}{k}-\frac{1}{k+d_y}+\frac{1} {(k+d_y) (k+d_y+2)}\right)c_2\right)c_1\,.
\end{eqnarray*}
\end{proof}

\begin{proof}[Proof of Theorem \ref{th4}]
The results of Theorem \ref{th4} are obtained by using Lemma \ref{lem3} with $\bb=C_t\bl$, $\bSigma=\bOmega$, $\bnu=\bmu$, $\mathbf{a}=r_{f,t+1}\bOne$ and Theorem \ref{th3}.
\begin{enumerate}[(a)]
\item In the case of the diffuse prior, using $k_y=n+k+1$, $d_y=n-k$, $v_y=n$, $\bm_y-\mathbf{a}=\overline{\bx}_{t,d}-r_{f,t+1}\bOne$, $\bS_y=\bS_{t,d}/n$, $c_1  = n C_t^2\bl^\top \bS_{t,d}^{-1} \bl$, $c_2  = n (\overline{\bx}_{t,d}-r_{f,t+1}\bOne)^\top \bS_{t,d}^{-1} (\overline{\bx}_{t,d}-r_{f,t+1}\bOne)$, and $c_{12} = n C_t\bl^\top \bS_{t,d}^{-1} (\overline{\bx}_{t,d}-r_{f,t+1}\bOne)$ we get
\begin{eqnarray*}
&&\mathbb{V}ar (\bl^\top \bw_t | \by)\\
&=&\frac{1}{n^2}\Bigg\{n(n+1)\Bigg[\left(1-\frac{2}{n}+\frac{2} {n (n+2)}\right)c_{12}^2
+\left(\frac{ n-k}{n(n+2)}+ \frac{1} {n (n+2)} c_2\right) c_1\Bigg]\\
&+& n\Bigg[\left(\frac{k-1}{n} +\left(1-\frac{1}{k}-\frac{1}{n}+\frac{1} {n (n+2)}\right)c_2\right)c_1
+\frac{2} {n (n+2)}c_{12}^2\Bigg]-(n-1)^2 c_{12}^2\Bigg\}\\
&=&\frac{n-1}{n^2}c_{12}^2+c_1\frac{1}{n^2} \left(\frac{n^2+k-2}{n+2}+\frac{n(k-1)}{k}c_2\right)\\
&=&\bl^\top\left(C_t^2\left( (n-1)\bS_{t,d}^{-1} (\overline{\bx}_{t,d}-r_{f,t+1}\bOne)(\overline{\bx}_{t,d}-r_{f,t+1}\bOne)^\top\bS_{t,d}^{-1}
+ \left(\frac{n^2+k-2}{n(n+2)}+\frac{k-1}{k}b_d\right)\bS_{t,d}^{-1} \right)\right)\bl
\end{eqnarray*}
where $b_d=n (\overline{\bx}_{t,d}-r_{f,t+1}\bOne)^\top \bS_{t,d}^{-1} (\overline{\bx}_{t,d}-r_{f,t+1}\bOne)$. Since $\bl$ is an arbitrary vector, the results in part (a) follow.

\item In the case of the conjugate prior, the application of $k_y=n+d_0+1$, $d_y=n+d_0-2k$, $v_y=n+r_0$, $\bm_y-\mathbf{a}=\overline{\bx}_{t,c}-r_{f,t+1}\bOne$, and $\bS_y=\bS_{t,c}/(n+r_0)$, $c_1  = (n+r_0)C_t^2\bl^\top \bS_{t,c}^{-1} \bl$, $c_2  = (n+r_0)(\overline{\bx}_{t,d}-r_{f,t+1}\bOne)^\top \bS_{t,c}^{-1} (\overline{\bx}_{t,d}-r_{f,t+1}\bOne)$, and $c_{12} = (n+r_0)C_t\bl^\top \bS_{t,c}^{-1} (\overline{\bx}_{t,c}-r_{f,t+1}\bOne)$.  leads to
\begin{eqnarray*}
&&\mathbb{V}ar (\bl^\top \bw_t | \by)\\
&=&\frac{1}{(n+r_0)^2}\Bigg\{(n+d_0-k)(n+d_0-k+1)\\
&\times&\Bigg[\left(1-\frac{2}{n+d_0-k}+\frac{2} {(n+d_0-k) (n+d_0-k+2)}\right)c_{12}^2\\
&+&\left(\frac{ n+d_0-2k}{(n+d_0-k)(n+d_0-k+2)}+ \frac{1} {(n+d_0-k) (n+d_0-k+2)} c_2\right) c_1\Bigg]\\
&+& (n+d_0-k)\Bigg[\left(\frac{k-1}{n+d_0-k} +\left(1-\frac{1}{k}-\frac{1}{n+d_0-k}+\frac{1} {(n+d_0-k) (n+d_0-k+2)}\right)c_2\right)c_1\\
&+&\frac{2} {(n+d_0-k) (n+d_0-k+2)}c_{12}^2\Bigg]-(n+d_0-k-1)^2 c_{12}^2\Bigg\}\\
&=&\frac{1}{(n+r_0)^2}\Bigg[\frac{n+d_0-k-1}{(n+d_0-k)^2}c_{12}^2+c_1\left(\frac{(n+d_0-k)^2+k-2}{n+d_0-k+2}+\frac{(n+d_0-k)(k-1)}{k}c_2\right)\Bigg]\\
&=&\bl^\top\Bigg\{C_t^2\Bigg[ (n+d_0-k-1)\bS_{t,c}^{-1} (\overline{\bx}_{t,c}-r_{f,t+1}\bOne)(\overline{\bx}_{t,c}-r_{f,t+1}\bOne)^\top\bS_{t,c}^{-1}\\
&+& \left(\frac{(n+d_0-k)^2+k-2}{(n+r_0)(n+d_0-k+2)}+\frac{(n+d_0-k)(k-1)}{(n+r_0)k}b_c\right)\bS_{t,c}^{-1} \Bigg]\Bigg\}\bl
\end{eqnarray*}
where $b_c=(n+r_0) (\overline{\bx}_{t,c}-r_{f,t+1}\bOne)^\top \bS_{t,c}^{-1} (\overline{\bx}_{t,c}-r_{f,t+1}\bOne)$. Since $\bl$ is an arbitrary vector, we get the statement of Theorem \ref{th4}.(b).
\end{enumerate}
\end{proof}

\begin{proof}[Proof of Theorem \ref{th5}]
Let $\bl$ be an arbitrary $k$-dimensional vector. From Theorem \ref{stochrep} with $\bL=\bl^\top$, we get the following stochastic representations of $\bL \bw_t$ under the diffuse prior and the conjugate prior expressed as
\begin{eqnarray*}
\bl^\top \bw_t &\stackrel{d}{=}& C_t \eta \bl^\top \bS_{t,d}^*(\bmu)^{-1} (\bmu-r_{f,t+1})
+ C_t \sqrt{ \eta } \Big((\bmu-r_{f,t+1})^\top \bS_{t,d}^*(\bmu)^{-1} (\bmu-r_{f,t+1}) \cdot \bl^\top \bS_{t,d}^*(\bmu)^{-1} \bl \\
&-& \bl^\top \bS_{t,d}^*(\bmu)^{-1} (\bmu-r_{f,t+1})(\bmu-r_{f,t+1})^\top \bS_{t,d}^*(\bmu)^{-1} \bl \Big)^{1/2} \bz_0,
\end{eqnarray*}
where $\eta  \sim \chi^2_{n}$, $\bz_0 \sim \mathcal N _p (\mathbf 0 , \mathbf I_p)$, and $\bmu | \bx \sim t_k\left(n-k,\overline{\bx}_{t,d},\bS_{t,d}/(n(n-k))\right)$, and \begin{eqnarray*}
\bl^\top \bw_t &\stackrel{d}{=}& C_t \eta \bl^\top \bS_{t,c}^*(\bmu)^{-1} (\bmu-r_{f,t+1})
+ C_t \sqrt{ \eta } \Big((\bmu-r_{f,t+1})^\top \bS_{t,c}^*(\bmu)^{-1} (\bmu-r_{f,t+1}) \cdot \bl^\top \bS_{t,c}^*(\bmu)^{-1} \bl \\
&-& \bl^\top\bS_{t,c}^*(\bmu)^{-1} (\bmu-r_{f,t+1})(\bmu-r_{f,t+1})^\top \bS_{t,c}^*(\bmu)^{-1} \bl \Big)^{1/2} \bz_0,
\end{eqnarray*}
where $\eta  \sim \chi^2_{n+d_0-k}$, $\bz_0 \sim \mathcal N _p (\mathbf 0 , \mathbf I_p)$, and $\bmu | \bx \sim t_k\left(n+d_0-2k,\overline{\bx}_{t,c},\bS_{t,c}/((n+r_0)(n+d_0-2k))\right)$.

Moreover, since
\[
\sqrt{n}
\left(
\left(
  \begin{array}{c}
    \eta /n \\
    \mathbf z_0/\sqrt{n} \\
    \bmu  \\
  \end{array}
\right)
-\left(
  \begin{array}{c}
    1 \\
    \mathbf 0 \\
    \overline{\bx}_{t,d}  \\
  \end{array}
\right)
\right)
\stackrel{d}{\longrightarrow} \mathcal{N}\left(\mathbf{0},
\left(
  \begin{array}{ccc}
    2 & \mathbf 0 & \mathbf 0 \\
    \mathbf 0 & \mathbf I_p & \mathbf 0 \\
    \mathbf 0 & \mathbf 0 & \breve{\bS}_{t}\\
  \end{array}
\right)
\right)
\]
and
\[
\sqrt{n}
\left(
\left(
  \begin{array}{c}
    \eta/n \\
    \mathbf z_0/\sqrt{n} \\
    \bmu \\
  \end{array}
\right)
-\left(
  \begin{array}{c}
    1 \\
    \mathbf 0 \\
    \overline{\bx}_{t,c}  \\
  \end{array}
\right)
\right)
\stackrel{d}{\longrightarrow} \mathcal{N}\left(\mathbf{0},
\left(
  \begin{array}{ccc}
    2 & \mathbf 0 & \mathbf 0 \\
    \mathbf 0 & \mathbf I_p & \mathbf 0 \\
    \mathbf 0 & \mathbf 0 & \breve{\bS}_{t} \\
  \end{array}
\right)
\right)\]
as $n \longrightarrow \infty$ as well as
\[\lim_{n \longrightarrow \infty} \overline{\bx}_{t,c}=\breve{\bx}_{t}=\lim_{n \longrightarrow \infty} \overline{\bx}_{t,d}\]
and
\[\lim_{n \longrightarrow \infty} \frac{\bS_{t,c}}{n+r_0}=\breve{\bS}_t=\lim_{n \longrightarrow \infty}\frac{\bS_{t,d}}{n-1},\]
the application of the delta method (c.f., \cite[Theorem 3.7]{dasGupta}) proves that
\[\sqrt{n}(\bl^\top\bw_{t}-\bl^\top\hat{\bw}_{t})|\bx_{t,n} \stackrel{d.}{\longrightarrow} \mathcal{N}_k(\mathbf{0},f_d)\]
and
\[\sqrt{n}(\bl^\top\bw_{t}-\bl^\top \hat{\bw}_{t})|\bx_{t,n} \stackrel{d.}{\longrightarrow} \mathcal{N}_k(\mathbf{0},f_c),\]
as $n \longrightarrow \infty$ under the diffuse prior and the conjugate prior, respectively.

Finally, the results of Theorem \ref{th4} yield
\begin{eqnarray*}
f_d&=&\lim_{n \longrightarrow \infty} \mathbb{V}ar(\sqrt{n}\bl^\top\bw_{t})
= \lim_{n \longrightarrow \infty} \bl^\top\Bigg\{C_t^2\Bigg( n(n-1)\bS_{t,d}^{-1} (\overline{\bx}_{t,d}-r_{f,t+1}\bOne)(\overline{\bx}_{t,d}-r_{f,t+1}\bOne)^\top\bS_{t,d}^{-1}\\
&+& \left(\frac{n^2+k-2}{n(n+2)}+\frac{k-1}{k}b_d\right)\bS_{t,d}^{-1}\Bigg) \Bigg\}\bl\\
&=& \bl^\top\Bigg\{C_t^2\Bigg[\breve{\bS}_{t}^{-1} (\breve{\bx}_{t}-r_{f,t+1}\bOne)(\breve{\bx}_{t}-r_{f,t+1}\bOne)^\top\breve{\bS}_{t}^{-1}\\
&+& \left(1+\frac{k-1}{k} (\breve{\bx}_{t}-r_{f,t+1}\bOne)^\top \breve{\bS}_{t}^{-1} (\breve{\bx}_{t}-r_{f,t+1}\bOne)  \right)\breve{\bS}_{t}^{-1}\Bigg] \Bigg\}\bl
\end{eqnarray*}
and, similarly,
\begin{eqnarray*}
f_c&=& \bl^\top\Bigg\{C_t^2\Bigg[\breve{\bS}_{t}^{-1} (\breve{\bx}_{t}-r_{f,t+1}\bOne)(\breve{\bx}_{t}-r_{f,t+1}\bOne)^\top\breve{\bS}_{t}^{-1}\\
&+& \left(1+\frac{k-1}{k} (\breve{\bx}_{t}-r_{f,t+1}\bOne)^\top \breve{\bS}_{t}^{-1} (\breve{\bx}_{t}-r_{f,t+1}\bOne)  \right)\breve{\bS}_{t}^{-1}\Bigg] \Bigg\}\bl=f_d.
\end{eqnarray*}
Since, for each $\bl$ the linear combination $\bl^\top\bw_t$ is asymptotically normally distributed, then we also get that the vector of weights $\bw_t$ is asymptotically normal.
\end{proof}

\begin{proof}[Proof of Theorem \ref{th6}]
Since $\bx_{t+1}|\bmu,\bSigma \sim \mathcal{N}_k(\bmu,\bSigma)$ and it is conditionally independent of $\bx_{t,n}$, we get
\begin{eqnarray*}
\widehat{W}_{t+1}|\bmu,\bSigma,\bx_{t,n} &\sim& \mathcal{N}(W_{t}(1+r_{f,t+1} + \bv_{t}^\top (\bmu-r_{f,t+1})),W^2_t \bv_{t}^\top \bSigma \bv_t).
\end{eqnarray*}

\begin{enumerate}[(a)]
\item In the case of the diffuse prior, we observe that
\begin{equation}\label{th6_eq1}
\frac{\bv_{t}^\top \bSigma \bv_t}{\bv_{t}^\top \bS_{t,d}(\bmu)^* \bv_t} \stackrel{d}{=} \frac{1}{\xi},
\end{equation}
where $\xi\sim \chi^2_{n-k+1}$ and is independent of $\bmu$ (see, e.g., Theorem 3.2.13 in \cite{Muirhead1982}). Then the stochastic representation of $\widehat{W}_{t+1}$ is given by
\begin{eqnarray*}
\widehat{W}_{t+1} &\stackrel{d}{=}& W_t\left(1+r_{f,t+1} + \bv_{t}^\top (\bmu-r_{f,t+1})+ \frac{\sqrt{\bv_{t}^\top \bS_{t,d}(\bmu)^* \bv_t}}{\sqrt{n-k+1}}t_2\right)\,,
\end{eqnarray*}
where $t_2\sim t_1(n-k+1,0,1)$ is independent of $\bmu$. Finally, from the properties of the multivariate $t$-distribution, we obtain
\[\bv_{t}^\top (\bmu-\overline{\bx}_{t,d}) \sim t_1\left(n-k,0,\frac{\bv_{t}^\top \bS_{t,d} \bv_t}{n(n-k)} \right)\,,\]
which leads to
\begin{eqnarray*}
\widehat{W}_{t+1} &\stackrel{d}{=}& W_t\Bigg(1+r_{f,t+1} + \bv_{t}^\top (\overline{\bx}_{t,d}-r_{f,t+1})\\
&+&\sqrt{\bv_{t}^\top \bS_{t,d} \bv_t}\left(\frac{t_1}{\sqrt{n(n-k)}}+ \sqrt{1+\frac{t_1^2}{n-k}}\frac{t_2}{\sqrt{n-k+1}}\right)\Bigg)\,,
\end{eqnarray*}
where $t_1$ and $t_2$ are independent with $t_1\sim t_{n-k}$ and $t_2\sim t_{n-k+1}$.

\item Similarly, for the conjugate prior, it holds that
\begin{equation}\label{th6_eq1}
\frac{\bv_{t}^\top \bSigma \bv_t}{\bv_{t}^\top \bS_{t,c}(\bmu)^* \bv_t} \stackrel{d}{=} \frac{1}{\xi},
\end{equation}
where $\xi\sim \chi^2_{n+d_0-2k+1}$ and is independent of $\bmu$. Then the stochastic representation of $\widehat{W}_{t+1}$ is given by
\begin{eqnarray*}
\widehat{W}_{t+1} &\stackrel{d}{=}& W_t\left(1+r_{f,t+1} + \bv_{t}^\top (\bmu-r_{f,t+1})+ \frac{\sqrt{\bv_{t}^\top \bS_{t,c}(\bmu)^* \bv_t}}{\sqrt{n+d_0-2k+1}}t_2\right)\,,
\end{eqnarray*}
where $t_2\sim t_{n+d_0-2k+1}$ is independent of $\bmu$. From the properties of the multivariate $t$-distribution, we get
\[\bv_{t}^\top (\bmu-\overline{\bx}_{t,c}) \sim t_1\left(n+d_0-2k,0,\frac{\bv_{t}^\top \bS_{t,c} \bv_t}{(n+r_0)(n+d_0-2k)} \right)\,,\]
which leads to
\begin{eqnarray*}
\widehat{W}_{t+1} &\stackrel{d}{=}& W_t\Bigg(1+r_{f,t+1} + \bv_{t}^\top (\overline{\bx}_{t,c}-r_{f,t+1})\\
&+&\sqrt{\bv_{t}^\top \bS_{t,c} \bv_t}\left(\frac{t_1}{\sqrt{(n+r_0)(n+d_0-2k)}}+ \sqrt{1+\frac{t_1^2}{n+d_0-2k}}\frac{t_2}{\sqrt{n+d_0-2k+1}}\right)\Bigg)\,,
\end{eqnarray*}
where $t_1$ and $t_2$ are independent with $t_1\sim t_{n+d_0-2k}$ and $t_2\sim t_{n+d_0-2k+1}$.
\end{enumerate}
\end{proof}

\subsection{Empirical Bayes estimation of the hyperparameters in the conjugate prior}

In this section, we derive the empirical Bayes estimates for the hyperparameters of the conjugate prior $\bm_0$ and $\bS_0$. Given the sample $\bx_{\tau,n}$ the empirical Bayes estimates for $\bm_0$ and $\bS_0$  are obtained by maximizing (see, e.g., \cite{carlin2000bayes})
\begin{equation}\label{g}
g(\bm_0,\bS_0)=\int_{\bmu} \int_{\bSigma} L(\bx_{t,n}|\bmu,\bSigma) \pi(\bmu,\bSigma) \mbox{d} \bSigma \mbox{d} \bmu
\end{equation}
with respect to $\bm_0$ and $\bS_0$.

First, we calculate the integral in \eqref{g}, ignoring the terms which do not depend on $\bm_0$ and $\bS_0$, to get
\begin{eqnarray*}
g(\bm_0,\bS_0)&\propto& \int_{\bmu} \int_{\bSigma} L(\bx_{t,n}|\bmu,\bSigma) \pi(\bmu,\bSigma) \mbox{d} \bSigma \mbox{d} \bmu\\
&\propto& \int_{\bmu} \int_{\bSigma} |\bSigma|^{-n/2} \exp \left\{-\frac{n}{2}(\bar\bx_{\tau}-\bmu)^\top \bSigma^{-1} (\bar\bx_{\tau}-\bmu) - \frac{n-1}{2} tr(\bS_{\tau} \bSigma^{-1})\right\} \\
&\times&|\bSigma|^{-1/2} \exp \left\{-\frac{r_0}{2}(\bmu-\bm_0)^\top \bSigma^{-1} (\bmu-\bm_0)-\right\}\\
&\times&|\bSigma|^{-d_0/2}|\bS_0|^{(d_0-k-1)/2} \exp\left\{ - \frac{1}{2} tr(\bS_{0} \bSigma^{-1})\right\} \mbox{d} \bSigma \mbox{d} \bmu\\
&=& |\bS_0|^{(d_0-k-1)/2} \int_{\bmu} \int_{\bSigma} |\bSigma|^{-(n+d_0+1)/2}\exp \left\{-\frac{1}{2}tr\left(\bSigma^{-1}\bV_{\tau}(\bmu;\bm_0,\bS_0)\right) \right\} \mbox{d} \bSigma \mbox{d} \bmu\\
&\propto& |\bS_0|^{(d_0-k-1)/2} \int_{\bmu} |\bV_{\tau}(\bmu;\bm_0,\bS_0)|^{-(n+d_0-k)/2} \mbox{d} \bmu\,,
\end{eqnarray*}
where the last identity is obtained by recognizing that under the integral with respect to $\bSigma$ we have a kernel of the density function of $\mathcal{IW}_k(n+d_0+1,\bV_{\tau}(\bmu;\bm_0,\bS_0))$ with $\bar{\by}_{\tau}(\bm_0)=(n \bar{\bx}_{\tau}+r_0\bm_0)/(n+r_0)$ and
\begin{eqnarray*}
&&\bV_{\tau}(\bmu;\bm_0,\bS_0)=\bS_{0} +(n-1)\bS_{\tau}+r_0(\bmu-\bm_0)(\bmu-\bm_0)^\top +n(\bar\bx_{\tau}-\bmu)(\bar\bx_{\tau}-\bmu)^\top\\
&=&\bS_{0} +(n-1)\bS_{\tau}+nr_0\frac{(\bm_0-\bar{\by}_{\tau}(\bm_0))(\bm_0-\bar{\by}_{\tau}(\bm_0))^\top}{n+r_0}+(n+r_0) (\bmu-\bar{\by}_{\tau}((\bm_0)))(\bmu-\bar{\by}_{\tau}(\bm_0))^\top
\,.
\end{eqnarray*}

Let $\widetilde{\bV}_{\tau}(\bm_0,\bS_0)=\bS_{0} +(n-1)\bS_{\tau}+nr_0(\bm_0-\bar{\by}_{\tau}(\bm_0))(\bm_0-\bar{\by}_{\tau}(\bm_0))^\top/(n+r_0)$. The application of Sylvester's determinant theorem leads to
\begin{eqnarray*}
|\bV_{\tau}(\bmu;\bm_0,\bS_0)|=|\widetilde{\bV}_{\tau}(\bm_0,\bS_0)|(1+(n+r_0) (\bmu-\bar{\by}_{\tau}(\bm_0))^\top\widetilde{\bV}_{\tau}(\bm_0,\bS_0)^{-1}(\bmu-\bar{\by}_{\tau}(\bm_0)))
\end{eqnarray*}
and, hence,
\begin{eqnarray*}
g(\bm_0,\bS_0)&\propto& |\bS_0|^{(d_0-k-1)/2} \int_{\bmu} |\bV_{\tau}(\bmu;\bm_0,\bS_0)|^{-(n+d_0-k)/2} \mbox{d} \bmu\\
&\propto& |\bS_0|^{(d_0-k-1)/2} |\widetilde{\bV}_{\tau}(\bm_0,\bS_0)|^{-(n+d_0-k)/2}\\
&\times& \int_{\bmu} (1+(n+r_0) (\bmu-\bar{\by}_{\tau}(\bm_0))^\top\widetilde{\bV}_{\tau}(\bm_0,\bS_0)^{-1}(\bmu-\bar{\by}_{\tau}(\bm_0)))^{-(n+d_0-k)/2} \mbox{d} \bmu\\
&\propto& |\bS_0|^{(d_0-k-1)/2} |\widetilde{\bV}_{\tau}(\bm_0,\bS_0)|^{-(n+d_0-k-1)/2} \\
&=& |\bS_0|^{(d_0-k-1)/2} |\bS_{0} +(n-1)\bS_{\tau}|^{-(n+d_0-k-1)/2}\\
&\times& \left(1+nr_0(\bm_0-\bar{\by}_{\tau}(\bm_0))^\top(\bS_{0} +(n-1)\bS_{\tau})^{-1}
(\bm_0-\bar{\by}_{\tau}(\bm_0))/(n+r_0)\right)^{-(n+d_0-k-1)/2}\,,
\end{eqnarray*}
where we use Sylvester's determinant theorem for the second time. From the last line, we conclude that $g(\bm_0,\bS_0)$ is maximized with respect to $\bm_0$ at $\hat{\bm}_0$ satisfying $\bm_0=\bar{\by}_{\tau}(\bm_0)$ independently of $\bS_0$ leading to $\hat{\bm}_0=\bar{\bx}_{\tau}$.

Taking the logarithms of $g(\bm_0,\bS_0)$, calculating the matrix derivative with respect to $\bS_0$ which is then set to the zero matrix, and substituting $\bm_0$ by $\hat{\bm}_0$, we get the following matrix equation
\[\frac{d_0-k-1}{2}\bS_0^{-1}- \frac{n+d_0-k-1}{2} (\bS_{0} +(n-1)\bS_{\tau})^{-1}=\mathbf{O} \]
with the solution given by
\[\hat{\bS}_0=\frac{(d_0-k-1)(n-1)}{n}\bS_{\tau}.\]

\bibliographystyle{econometrica}
\bibliography{EXP-BA}

\end{document}